\documentclass[a4paper, reqno]{amsart}
\usepackage{geometry}
\usepackage{amssymb}
\usepackage{fancyhdr}
\usepackage{mathrsfs}
\usepackage{mathtools}
\usepackage{stmaryrd}
\usepackage{amsmath}
\usepackage{tikz-cd}
\usepackage{tikz}
\usepackage{pifont}
\usepackage{color}
\usepackage{caption}
\usepackage[all]{xy}
\usepackage{type1cm}
\usepackage{hyperref}
\usepackage{color}
\hypersetup{colorlinks=true,
     breaklinks=true,
     linkcolor=blue,
      pageanchor=true}

\newtheorem{thm}{Theorem}[section]

\newtheorem{cor}[thm]{Corollary}
\newtheorem{lem}[thm]{Lemma}
\newtheorem{prop}[thm]{Proposition}

\newtheorem{claim}[thm]{Claim}
\theoremstyle{definition}
\newtheorem{defn}[thm]{Definition}
\newtheorem{exam}[thm]{Example}
\newtheorem{rem}[thm]{Remark}

\title[Auslander-Reiten-Serre duality revisited]{Auslander-Reiten-Serre duality revisited}
\author{Ji-Wei He and Jixing Pan}
\subjclass{16G10, 16G70, 18G80.}
\keywords{Auslander-Reiten-Serre duality, extriangulated category, relative structure, exact functor.}
\address{Ji-Wei He: School of Mathematics, Hangzhou Normal University, Hangzhou Zhejiang 311121, PR China}
\email{jwhe@hznu.edu.cn}
\address{J. Pan: School of Mathematics, Nanjing University, Nanjing Jiangsu 210093, PR China.\\
	School of Mathematics, Hangzhou Normal University, Hangzhou Zhejiang 311121, PR China}
\email{20243024@hznu.edu.cn}

\begin{document}

\begin{abstract}
	Let $(\mathcal{C},\mathbb{E},\mathfrak{s})$ be an extriangulated category with a right Auslander-Reiten-Serre (ARS for short) duality $(\tau, \eta)$ in the sense of Iyama, Nakaoka and Palu. We show $(\tau, \eta)$ induces right ARS dualities on the relative theories of $(\mathcal{C},\mathbb{E},\mathfrak{s})$. Under relative structures, we show that the functor $\tau$ is indeed an exact functor between extriangulated categories and preserves almost split exangles. Finally, we give some applications and examples on these results. For instance, we generalize a recent result by A. Hubery to a categorical framework.
\end{abstract}

\maketitle

\section{Introduction}

Auslander-Reiten theory plays an important role in representation theory of algebras. It started in 1970s and is still an active area today. As the notion of extriangulated categories was invented a decade ago by Nakaoka and Palu \cite{NP} to unify exact categories and triangulated categories, Iyama, Nakaoka and Palu \cite{INP} generalized the classical Auslander-Reiten theory to extriangulated categories. One of the advantages of using extriangulated category is that it is closed under several operations such as taking extension closed subcategories \cite{NP}, ideal quotients by subcategories consisting of projective-injective objects \cite{NP} and relative theories \cite{HLN,ASo,DRSSK}. This allows us to break the barriers between different theories and work under a unified framework. Let $R$ be a commutative artin ring and $(\mathcal{C},\mathbb{E},\mathfrak{s})$ be an $R$-linear extriangulated category. In \cite{INP}, the authors introduced the notion of right Auslander-Reiten-Serre (ARS) dualities,  that is, a pair $(\tau,\eta)$ in which $\tau:\underline{\mathcal{C}}\rightarrow \overline{\mathcal{C}}$ is an $R$-linear functor and $\eta_{A,B}:\underline{\mathcal{C}}(A,B)\simeq \mathbb{D}\mathbb{E}(B,\tau A)$ is a binatural isomorphism. If $\tau$ is an equivalence, then the pair is called an ARS duality. They proved that under certain conditions, the existence of ARS duality is equivalent to that of almost split extensions \cite[Theorem 3.6]{INP}. They also discussed the induced almost split extensions in relative theories, ideal quotients and extension-closed subcategories. Inspired by their work, we mainly focus on right ARS dualities instead of almost split extensions. Suppose $(\mathcal{C},\mathbb{E},\mathfrak{s})$ is Ext-finite and has a right ARS duality $(\tau,\eta)$. First, we discuss the induced right ARS dualities by relative theories.

\begin{thm}(Theorem \ref{main thm1})
	Let $\mathbb{F}\xhookrightarrow{\iota} \mathbb{E}$ be a closed subfunctor. Then $(\tau,\eta)$ induces a right ARS duality $(\tau_{\mathbb{F}},\eta^{\mathbb{F}})$ on the relative theory $(\mathcal{C},\mathbb{F},\mathfrak{s}|_{\mathbb{F}})$. Moreover, there are commutative diagrams as follows for any $A,B\in \mathcal{C}$.
	\[\begin{tikzcd}
		\underline{\mathcal{C}} \arrow[r,"\tau"] \arrow[d,"\pi"] & \overline{\mathcal{C}} \arrow[d,"\pi'"] \\
		\mathcal{C}/\mathcal{P}_{\mathbb{F}} \arrow[r,"\tau_{\mathbb{F}}"] & \mathcal{C}/\mathcal{I}_{\mathbb{F}}
	\end{tikzcd}\text{ and }
	\begin{tikzcd}
		\underline{\mathcal{C}}(A,B) \arrow[r,"\eta_{A,B}","\sim" swap] \arrow[d,"\pi"] & \mathbb{D}\mathbb{E}(B,\tau A) \arrow[d,"\mathbb{D}(\iota)"] \\
		\mathcal{C}/\mathcal{P}_{\mathbb{F}}(A,B) \arrow[r,"\eta^{\mathbb{F}}_{A,B}","\sim" swap] & \mathbb{D}\mathbb{F}(B,\tau_{\mathbb{F}} A)
	\end{tikzcd}\]
\end{thm}

This theorem enables us to deduce the Auslander-Reiten translations $\tau_{[m]}$ in $m$-extended module categories introduced by Zhou \cite{Zhou} in a different way (Section \ref{examples}). 

By investigating the extriangulated structure in the above theorem, we obtain the following result. It is analogous to the property of Serre functors in triangulated categories \cite{BK}. Moreover, we remark that some tricks used in the proof originate from the work by M. Van den Bergh \cite[Theorem A.4.4]{Bo}.

\begin{thm}(Theorem \ref{main thm2})
	Suppose $(\mathcal{C},\mathbb{E},\mathfrak{s})$ is weakly idempotent complete and has enough projectives $\mathcal{Q}$ and enough injectives $\mathcal{J}$ which are both functorially finite. Define
	\[(\underline{\mathcal{C}},\underline{\mathbb{E}^{\mathcal{Q}}},\underline{\mathfrak{s}|_{\mathbb{E}^{\mathcal{Q}}}})\,(\text{resp. }(\overline{\mathcal{C}},\overline{\mathbb{E}_{\mathcal{J}}},\overline{\mathfrak{s}|_{\mathbb{E}_{\mathcal{J}}}}))\]
	to be the ideal quotient of the relative theory $(\mathcal{C},\mathbb{E}^{\mathcal{Q}},\mathfrak{s}|_{\mathbb{E}^{\mathcal{Q}}})\,(\text{resp. } (\mathcal{C},\mathbb{E}_{\mathcal{J}},\mathfrak{s}|_{\mathbb{E}_{\mathcal{J}}}))$ by $\mathcal{Q}\,(\text{resp. } \mathcal{J})$. Then there exists a natural isomorphism $\phi:\underline{\mathbb{E}^{\mathcal{Q}}}\xRightarrow{\sim} \overline{\mathbb{E}_{\mathcal{J}}}\circ (\tau^{\rm op}\times \tau)$ such that
	\[(\tau,\phi):(\underline{\mathcal{C}},\underline{\mathbb{E}^{\mathcal{Q}}},\underline{\mathfrak{s}|_{\mathbb{E}^{\mathcal{Q}}}})\rightarrow (\overline{\mathcal{C}},\overline{\mathbb{E}_{\mathcal{J}}},\overline{\mathfrak{s}|_{\mathbb{E}_{\mathcal{J}}}})\]
	is an exact functor which is fully faithful.
\end{thm}

We show the functor $\tau$ also preserves almost split extensions (Proposition \ref{tau preserve almost split extension}) even though it may not be an equivalence.

Finally, we give some applications and examples. Recently, Hubery \cite{H} investigated the invariance of the Auslander-Reiten formula under the Auslander-Reiten translation for hereditary algebras. As an application of our results, we generalize it to the following one.

\begin{thm}(Theorem \ref{invariance})
	Under the setting of the above theorem, consider the non-degenerate bilinear form
	\[\{-,-\}: \mathbb{E}(C,A)\times \overline{\mathcal{C}}(A,\tau C)\rightarrow J\]
	given by $\eta$. Then for any $\delta \in \mathbb{E}^{\mathcal{Q}}(C,A)$ and $f\in \mathcal{C}(A,\tau C)$, we have
	\[\{\delta,\overline{f}\}=\{\phi(\delta),\tau(\underline{f})\}\]
\end{thm}

If $\tau$ is an equivalence, we show that under certain conditions it induces a triangulated auto-equivalence on a triangulated category. For example in the category of finitely generated modules over an artin algebra, if the $\tau$-orbits of indecomposable projective modules contain all indecomposable injective modules, then the ideal quotient by the additive subcategory generated by these $\tau$-orbits is a triangulated category and $\tau$ induces a triangulated auto-equivalence on it (Corollary \ref{artin alg case}). This is analogous to self-injective algebras.

This paper is organized as follows. In Section 2 we recall some related notions. In Section 3, we investigate fundamental properties of right ARS dualities and prove our first main theorem. Section 4 is devoted to prove our second main theorem. In Section 5 we give some applications and examples. For simplicity of the results, we put settings at the beginning of each section.

\section{Preliminaries}

Throughout this article, let $R$ be an arbitrary commutative unital ring unless otherwise specified. All extriangulated categories are assumed to be $R$-linear in the sense of Definition \ref{def of extri}.

\subsection{Extriangulated categories}

We briefly recall the definition of $R$-linear extriangulated categories. For details, we refer the reader to \cite{NP}. Let $\mathcal{C}$ be an $R$-linear additive category and $\mathbb{E}:\mathcal{C}^{\rm op}\times \mathcal{C}\rightarrow {\rm Mod}R$ be an $R$-bilinear functor. For $\delta \in \mathbb{E}(C,A)$ and $c:C'\rightarrow C,\, a:A\rightarrow A'$, we denote $\mathbb{E}(c,A)(\delta)$ and $\mathbb{E}(C,a)(\delta)$ by $c^{*}\delta$ and $a_{*}\delta$ respectively. By Yoneda's Lemma, $\delta \in \mathbb{E}(C,A)$ induces two natural transformations $\delta^{\sharp}:\mathcal{C}(A,-)\Rightarrow \mathbb{E}(C,-)$ and $\delta_{\sharp}:\mathcal{C}(-,C)\Rightarrow \mathbb{E}(-,A)$.

\begin{defn}\cite[Definition 2.12]{NP}\label{def of extri}
	Let $R$ be a commutative unital ring and $\mathcal{C}$ be an $R$-linear additive category. A triplet $(\mathcal{C},\mathbb{E},\mathfrak{s})$ is called an {\em (R-linear) extriangulated category} if it satisfies the following conditions.
	\begin{enumerate}
		\item [(ET1)] $\mathbb{E}:\mathcal{C}^{\rm op}\times \mathcal{C}\rightarrow {\rm Mod}R$ is an $R$-bilinear functor.
		\item [(ET2)] $\mathfrak{s}$ is an additive realization of $\mathbb{E}$.
		\item [(ET3)] Let $\delta \in \mathbb{E}(C,A)$ and $\delta' \in \mathbb{E}(C',A')$ be $\mathbb{E}$-extensions, realized as
		\[\mathfrak{s}(\delta)=[A\xrightarrow{x} B\xrightarrow{y} C] \text{ and } \mathfrak{s}(\delta')=[A'\xrightarrow{x'} B'\xrightarrow{y'} C'].\]
		For any commutative square
		\[\begin{tikzcd}
			A \arrow[d,"a",swap] \arrow[r,"x"] & B \arrow[d,"b"] \arrow[r,"y"] & C \\
			A' \arrow[r,"x'"] & B' \arrow[r,"y'"] & C'
		\end{tikzcd}\]
		in $\mathcal{C}$, there exists a morphism $(a,c):\delta \rightarrow \delta'$ such that $cy=y'b$.
		\item [${\rm (ET3)^{op}}$] Dual of (ET3).
		\item [(ET4)] Let $\delta \in \mathbb{E}(D,A)$ and $\delta' \in \mathbb{E}(F,B)$ be $\mathbb{E}$-extensions, realized by
		\[A\xrightarrow{f} B\xrightarrow{f'} D\text{ and }B\xrightarrow{g} C\xrightarrow{g'} F\]
		respectively. Then there exist an object $E\in \mathcal{C}$, a commutative diagram
		\[\begin{tikzcd}
			A \arrow[r,"f"] \arrow[d,equal] & B \arrow[r,"f'"] \arrow[d,"g"] & D \arrow[d,"d"] \\
			A \arrow[r,"h"] & C \arrow[r,"h'"] \arrow[d,"g'"] & E \arrow[d,"e"] \\
			& F \arrow[r,equal] & F
		\end{tikzcd}\]
		in $\mathcal{C}$, and an $\mathbb{E}$-extension $\delta''\in \mathbb{E}(E,A)$ realized by $A\xrightarrow{h} C\xrightarrow{h'} E$, which satisfy the following compatibilities.
		\begin{enumerate}
			\item [(i)] $D\xrightarrow{d} E\xrightarrow{e} F$ realizes $f'_{*}\delta'$.
			\item [(ii)] $d^{*}\delta''=\delta$.
			\item [(iii)] $f_{*}\delta''=e^{*}\delta'$.
		\end{enumerate}
		\item [${\rm (ET4)^{op}}$] Dual of (ET4).
	\end{enumerate}
\end{defn}

In the rest of this section, let $(\mathcal{C},\mathbb{E},\mathfrak{s})$ be an $R$-linear extriangulated category. For $\delta \in \mathbb{E}(C,A)$, if $\mathfrak{s}(\delta)=[A\xrightarrow{a} B\xrightarrow{b} C]$, then we call $A\xrightarrow{a} B\xrightarrow{b} C\stackrel{\delta}\dashrightarrow$ an $\mathfrak{s}$-triangle (to emphasize $\mathfrak{s}$) or simply exangle.

\begin{lem}\cite[Corollary 3.12]{NP}\label{exact seq}
	For any $\mathfrak{s}$-triangle $A\xrightarrow{x} B\xrightarrow{y} C\stackrel{\delta}\dashrightarrow $, the following sequences of natural transformations are exact.
	\[\mathcal{C}(C,-)\xrightarrow{\mathcal{C}(y,-)} \mathcal{C}(B,-)\xrightarrow{\mathcal{C}(x,-)} \mathcal{C}(A,-)\xrightarrow{\delta^{\sharp}} \mathbb{E}(C,-)\xrightarrow{y^{*}} \mathbb{E}(B,-)\xrightarrow{x^{*}} \mathbb{E}(A,-)\]
	\[\mathcal{C}(-,A)\xrightarrow{\mathcal{C}(-,x)} \mathcal{C}(-,B)\xrightarrow{\mathcal{C}(-,y)} \mathcal{C}(-,C)\xrightarrow{\delta_{\sharp}} \mathbb{E}(-,A)\xrightarrow{x_{*}} \mathbb{E}(-,B)\xrightarrow{y_{*}} \mathbb{E}(-,C)\]
\end{lem}

\begin{defn}
	An object $P\in \mathcal{C}$ is called {\em projective} if $\mathbb{E}(P,-)=0$. We say $(\mathcal{C},\mathbb{E},\mathfrak{s})$ {\em has enough projectives} if for any object $C\in \mathcal{C}$, there exists an $\mathfrak{s}$-triangle $A\rightarrow P\rightarrow C\dashrightarrow $ such that $P$ is projective. Dually, we define {\em injective} object.
\end{defn}

\begin{lem}\cite[Proposition 3.30]{NP}\label{induced extri str on ideal quotient}
	Let $\mathcal{D}\subseteq \mathcal{C}$ be a full additive subcategory in which all objects are both projective and injective. Denote by $\widetilde{\mathcal{C}}$ the ideal quotient $\mathcal{C}/[\mathcal{D}]$ and by $\widetilde{\mathbb{E}}:\widetilde{\mathcal{C}}^{\rm op}\times \widetilde{\mathcal{C}}\rightarrow {\rm Mod}R$ the $R$-bilinear functor induced by $\mathbb{E}$. Define $\widetilde{\mathfrak{s}}(\delta):=\widetilde{\mathfrak{s}(\delta)}=[A\xrightarrow{\widetilde{a}} B\xrightarrow{\widetilde{b}} C]$ for any $\delta \in \widetilde{\mathbb{E}}(C,A)$ and $\mathfrak{s}(\delta)=[A\xrightarrow{a} B\xrightarrow{b} C]$. Then $(\widetilde{\mathcal{C}},\widetilde{\mathbb{E}},\widetilde{\mathfrak{s}})$ is an extriangulated category.
\end{lem}

\begin{lem}\cite[Proposition 1.20]{LN}\label{weak pushout}
	Let $A\xrightarrow{x} B\xrightarrow{y} C\stackrel{\delta}\dashrightarrow$ be an $\mathfrak{s}$-triangle and $f:A\rightarrow D$ be any morphism. Suppose $f_{*}\delta$ is realized by $D\xrightarrow{d} E\xrightarrow{e} C\stackrel{f_{*}\delta}\dashrightarrow $. Then there exists a morphism $g:B\rightarrow E$ which gives a morphism of $\mathfrak{s}$-triangles
	\[\begin{tikzcd}
		A \arrow[r,"x"] \arrow[d,"f"] & B \arrow[r,"y"] \arrow[d,"g",dashed] & C \arrow[r,"\delta",dashed] \arrow[d,equal] & {} \\
		D \arrow[r,"d"] & E \arrow[r,"e"] & C \arrow[r,"f_{*}\delta",dashed] & {}
	\end{tikzcd}\]
	and an $\mathfrak{s}$-triangle $A\xrightarrow{\begin{pmatrix}
			\begin{smallmatrix}
				-f\\
				x
			\end{smallmatrix}
	\end{pmatrix}} D\oplus B\xrightarrow{\begin{pmatrix}
	\begin{smallmatrix}
		d & g 
	\end{smallmatrix}
\end{pmatrix}} E\stackrel{e^{*}\delta}\dashrightarrow $.
\end{lem}

\subsection{Exact functors}

\begin{defn}
	Let $(\mathcal{C},\mathbb{E},\mathfrak{s}),(\mathcal{C}',\mathbb{E}',\mathfrak{s}')$ be $R$-linear extriangulated categories.
	\begin{enumerate}
		\item \cite[Definition 2.32]{BS} An {\em exact functor} $(F,\phi):(\mathcal{C},\mathbb{E},\mathfrak{s})\rightarrow (\mathcal{C}',\mathbb{E}',\mathfrak{s}')$ is a pair of an $R$-linear functor $F:\mathcal{C}\rightarrow \mathcal{C}'$ and a natural transformation $\phi:\mathbb{E}\Rightarrow \mathbb{E}'\circ (F^{\rm op}\times F)$ which satisfies
		\[\mathfrak{s}'(\phi_{C,A}(\delta))=[F(A)\xrightarrow{F(x)} F(B)\xrightarrow{F(y)} F(C)]\]
		for any $\mathfrak{s}$-triangle $A\xrightarrow{x} B\xrightarrow{y} C\stackrel{\delta} \dashrightarrow $ in $\mathcal{C}$.
		\item \cite[Definition 2.11]{NOS} Let $(F,\phi),(G,\psi):(\mathcal{C},\mathbb{E},\mathfrak{s})\rightarrow (\mathcal{C}',\mathbb{E}',\mathfrak{s}')$ be exact functors. A {\em natural transformation $\eta:(F,\phi)\Rightarrow (G,\psi)$ of exact functors} is a natural transformation $\eta:F\Rightarrow G$ of $R$-linear functors which satisfies
		\[(\eta_{A})_{*}\phi_{C,A}(\delta)=(\eta_{C})^{*}\psi_{C,A}(\delta)\]
		for any $\delta \in \mathbb{E}(C,A)$.
	\end{enumerate}
\end{defn}

\begin{rem}\label{rem on exact functors}
	(1) Let $(F,\phi):(\mathcal{C},\mathbb{E},\mathfrak{s})\rightarrow (\mathcal{C}',\mathbb{E}',\mathfrak{s}')$ be an exact functor. By \cite[Proposition 2.13]{NOS}, $F$ is an equivalence and $\phi$ is a natural isomorphism if and only if there exists an exact functor $(G,\psi):(\mathcal{C}',\mathbb{E}',\mathfrak{s}')\rightarrow (\mathcal{C},\mathbb{E},\mathfrak{s})$ such that $(G,\psi)\circ (F,\phi)\cong ({\rm id}_{\mathcal{C}},{\rm id}_{\mathbb{E}})$ and $(F,\phi)\circ (G,\psi)\cong ({\rm id}_{\mathcal{C}'},{\rm id}_{\mathbb{E}'})$ as exact functors. In this case, we call $(F,\phi)$ an {\em equivalence of extriangulated categories} or simply {\em exact equivalence}.
	
	(2) All the above notions coincide with the usual ones if all extriangulated categories are exact categories or triangulated categories (see \cite[Theorem 2.33, 2.34]{BS}).
\end{rem}

\subsection{Relative theories}

Let $\mathbb{F}$ be an $R$-bilinear subfunctor of $\mathbb{E}$. Denote by $\mathfrak{s}|_{\mathbb{F}}$ the restriction of $\mathfrak{s}$ to $\mathbb{F}$. Then $(\mathcal{C},\mathbb{F},\mathfrak{s}|_{\mathbb{F}})$ satisfies (ET1), (ET2), (ET3) and ${\rm (ET3)^{op}}$. 

\begin{lem}\cite[Proposition 3.16]{HLN}
	The following are equivalent
	\begin{enumerate}
		\item $(\mathcal{C},\mathbb{F},\mathfrak{s}|_{\mathbb{F}})$ is extriangulated.
		\item $\mathfrak{s}|_{\mathbb{F}}$-inflations are closed under compositions.
		\item $\mathfrak{s}|_{\mathbb{F}}$-deflations are closed under compositions.
	\end{enumerate}
\end{lem}

If $\mathbb{F}$ satisfies the above equivalent conditions, we call it a {\em closed subfunctor} of $\mathbb{E}$. We also call the pair $(\mathbb{F},\mathfrak{s}|_{\mathbb{F}})$ a {\em relative structure} or {\em substructure} of the extriangulated structure $(\mathbb{E},\mathfrak{s})$ on $\mathcal{C}$, and call $(\mathcal{C},\mathbb{F},\mathfrak{s}|_{\mathbb{F}})$ a {\em relative theory} of $(\mathcal{C},\mathbb{E},\mathfrak{s})$.

\begin{exam}\cite[Definition 3.18, Proposition 3.19]{HLN}\label{typical exam of subfunctor}
	This is a typical example of constructing relative structures. Let $\mathcal{D}\subseteq \mathcal{C}$ be a full subcategory. Define
	\[\mathbb{E}_{\mathcal{D}}(C,A)=\{\delta \in \mathbb{E}(C,A)\,|\,(\delta_{\sharp})_{D}=0\text{ for any }D\in \mathcal{D}\},\]
	\[\mathbb{E}^{\mathcal{D}}(C,A)=\{\delta \in \mathbb{E}(C,A)\,|\,(\delta^{\sharp})_{D}=0\text{ for any }D\in \mathcal{D}\}.\]
	Then $\mathbb{E}_{\mathcal{D}}$ and $\mathbb{E}^{\mathcal{D}}$ are closed subfunctors of $\mathbb{E}$. Moreover, $\mathbb{E}_{\mathcal{D}}$ (resp. $\mathbb{E}^{\mathcal{D}}$) is the greatest closed subfunctor of $\mathbb{E}$ such that objects in $\mathcal{D}$ are projective (resp. injective) under this substructure.
\end{exam}

\subsection{Auslander-Reiten-Serre duality}

\begin{defn}\cite[Definition 1.21]{INP}
	Let $\mathcal{P}_{\mathbb{E}}$ (resp. $\mathcal{I}_{\mathbb{E}}$) denote the ideal of $\mathcal{C}$ consisting of all morphisms $f$ satisfying $\mathbb{E}(f,-)=0$ (resp. $\mathbb{E}(-,f)=0$). The {\em stable category} (resp. {\em costable category}) of $\mathcal{C}$ is defined to be the ideal quotient
	\[\underline{\mathcal{C}}:=\mathcal{C}/\mathcal{P}_{\mathbb{E}}\text{ (resp. $\overline{\mathcal{C}}:=\mathcal{C}/\mathcal{I}_{\mathbb{E}}$)}.\]
\end{defn}

Note that if $(\mathcal{C},\mathbb{E},\mathfrak{s})$ has enough projectives $\mathcal{Q}$ (resp. injectives $\mathcal{J}$), then $\mathcal{P}_{\mathbb{E}}$ (resp. $\mathcal{I}_{\mathbb{E}}$) coincides with the ideal $[\mathcal{Q}]$ (resp. $[\mathcal{J}]$) consisting of morphisms factoring through projective (resp. injective) objects. 

\begin{lem}\cite[Proposition 1.23]{INP}\label{(co)stable cat}
	\begin{enumerate}
		\item The functor $\mathbb{E}:\mathcal{C}^{\rm op}\times \mathcal{C}\rightarrow {\rm Mod}R$ induces a functor $\mathbb{E}:\underline{\mathcal{C}}^{\rm op}\times \overline{\mathcal{C}}\rightarrow {\rm Mod}R$.
		\item An object $B\in \mathcal{C}$ is injective if and only if $B\cong 0$ in $\overline{\mathcal{C}}$.
		\item An object $A\in \mathcal{C}$ is projective if and only if $A\cong 0$ in $\underline{\mathcal{C}}$.
	\end{enumerate}
\end{lem}

In the rest, we assume $R$ is a commutative artin ring. Denote by mod$R$ the finitely generated $R$-module and $J$ its minimal injective cogenerator. Recall the Matlis duality $\mathbb{D}:={\rm Hom}_{R}(-,J)$ on mod$R$ (cf. \cite[Chapter II]{ARS}).

\begin{defn}\cite[Definition 3.4]{INP}
	Let $(\mathcal{C},\mathbb{E},\mathfrak{s})$ be an $R$-linear extriangulated category
	\begin{enumerate}
		\item A {\em right Auslander-Reiten-Serre (ARS) duality} is a pair $(\tau,\eta)$ of an $R$-linear functor $\tau: \underline{\mathcal{C}}\rightarrow \overline{\mathcal{C}}$ and a binatural isomorphism of $R$-modules
		\[\eta_{A,B}:\underline{\mathcal{C}}(A,B)\simeq \mathbb{D}\mathbb{E}(B,\tau A)\text{ for any }A,B\in \mathcal{C}.\]
		\item If moreover $\tau$ is an equivalence, we say that $(\tau,\eta)$ is an {\em Auslander-Reiten-Serre (ARS) duality}.
	\end{enumerate}
\end{defn}

Dually, there is the notion of {\em left ARS duality}. Note that if $(\tau,\eta)$ is a right ARS duality, then $\tau$ is fully faithful (see \cite[Remark 3.5]{INP}). The existence of ARS duality is closely related to the existence of {\em almost split extensions} \cite[Theorem 3.6]{INP}.

\section{Induced Auslander-Reiten-Serre duality}

In this section, we show that a right ARS duality of an extriangulated category induces right ARS dualities of the relative theories. Throughout, let $R$ be a commutative artin ring and $(\mathcal{C},\mathbb{E},\mathfrak{s})$ be an $R$-linear {\em Ext-finite} extriangulated category, that is, $\mathbb{E}(C,A)$ is a finitely generated $R$-module for any $C,A\in \mathcal{C}$. First, we show that there is an equivalent definition for right ARS dualities under Ext-finiteness.

\begin{lem}\label{equi def}
	Let $\tau: \underline{\mathcal{C}}\rightarrow \overline{\mathcal{C}}$ be an $R$-linear functor. The following are equivalent.
	\begin{enumerate}
		\item There is a binatural isomorphism of $R$-modules $\underline{\mathcal{C}}(A,B)\simeq \mathbb{D}\mathbb{E}(B,\tau A)$ for any $A,B\in \mathcal{C}$.
		\item There is a binatural isomorphism of $R$-modules $\mathbb{E}(A,B)\simeq \mathbb{D}\overline{\mathcal{C}}(B,\tau A)$ for any $A,B\in \mathcal{C}$.
	\end{enumerate}
\end{lem}
\begin{proof}
	We only show (1) implies (2), the other direction is similar. Let $\eta_{A,B}$ denote the isomorphism in (1). Define $\eta_{A}:=\eta_{A,A}(\underline{{\rm id}_{A}})$. For any $A,B\in \mathcal{C}$, define an $R$-bilinear form
	\[\mathbb{E}(A,B)\times \overline{\mathcal{C}}(B,\tau A)\longrightarrow J\]
	by $(\delta,\overline{f})\longmapsto \eta_{A}(\overline{f}_{*}\delta)$. We show it is non-degenerate. Assume $\overline{f}\neq 0$. By definition, there exists $X\in \mathcal{C}$ satisfying $\mathbb{E}(X,f)\neq 0$. Then there exists an $\mathbb{E}$-extension $\alpha \in \mathbb{E}(X,B)$ such that $\overline{f}_{*}\alpha \in \mathbb{E}(X,\tau A)$ is nonzero. By Ext-finiteness, $\overline{f}_{*}\alpha \in \mathbb{E}(X,\tau A)\simeq \mathbb{D}\underline{\mathcal{C}}(A,X)$ corresponds to a nonzero element $h\in \mathbb{D}\underline{\mathcal{C}}(A,X)$. Thus there exists $\underline{c}\in \underline{\mathcal{C}}(A,X)$ satisfying $h(\underline{c})\neq 0$. Consider the following commutative diagram
	\[\begin{tikzcd}
		\mathbb{E}(X,\tau A) \arrow[r,"\underline{c}^{*}"] \arrow[d,"\simeq"] & \mathbb{E}(A,\tau A) \arrow[d,"\simeq"] \\
		\mathbb{D}\underline{\mathcal{C}}(A,X) \arrow[r,"\mathbb{D}(\underline{c}_{*})"] & \mathbb{D}\underline{\mathcal{C}}(A,A).
	\end{tikzcd}\]
	Let $\delta:=\underline{c}^{*}\alpha \in \mathbb{E}(A,B)$, then $(\delta,\overline{f})\longmapsto \eta_{A}(\overline{f}_{*}\delta)=\eta_{A}(\underline{c}^{*}\overline{f}_{*}\alpha)=\mathbb{D}(\underline{c}_{*})(h)(\underline{{\rm id}_{A}})=h(\underline{c})\neq 0$. On the other hand, assume $\delta \in \mathbb{E}(A,B)$ is non-split. Then in $\mathfrak{s}$-triangle $B\xrightarrow{a} C\xrightarrow{b} A\stackrel{\delta}\dashrightarrow $, the morphism $b$ is a non-retraction. Consider the following commutative diagram
	\[\begin{tikzcd}
		\mathbb{E}(A,\tau A) \arrow[r,"\underline{b}^{*}"] \arrow[d,"\simeq"] & \mathbb{E}(C,\tau A) \arrow[d,"\simeq"] \\
		\mathbb{D}\underline{\mathcal{C}}(A,A) \arrow[r,"\mathbb{D}(\underline{b}_{*})"] & \mathbb{D}\underline{\mathcal{C}}(A,C).
	\end{tikzcd}\]
	By assumption and Lemma \ref{(co)stable cat}, the $R$-module $\underline{\mathcal{C}}(A,A)$ is nonzero. Since $b$ is a non-retraction, the homomorphism $\underline{\mathcal{C}}(A,C)\xrightarrow{\underline{b}_{*}} \underline{\mathcal{C}}(A,A)$ is not a surjection. Indeed, assume otherwise, there exists $\underline{b'}:A\rightarrow C$ such that $\underline{bb'}=\underline{{\rm id}_{A}}$. Then ${\rm id}_{A}-bb'\in \mathcal{P}_{\mathbb{E}}$, which implies the existence of $d:A\rightarrow C$ such that ${\rm id}_{A}=b(b'+d)$, a contradiction. Let $M:={\rm coker}\,\underline{b}_{*}$. The image $[\underline{{\rm id}_{A}}]$ of $\underline{{\rm id}_{A}}$ in $M$ is nonzero. Then there exists a nonzero $g\in \mathbb{D}M\subseteq \mathbb{D}\underline{\mathcal{C}}(A,A)$ such that $g([\underline{{\rm id}_{A}}])\neq 0$. Let $\rho \in \mathbb{E}(A,\tau A)$ be the $\mathbb{E}$-extension corresponding to $g$. Then $\underline{b}^{*}\rho =0$ holds by the above commutative diagram. Therefore we obtain a morphism of $\mathfrak{s}$-triangles by ${\rm (ET3)^{op}}$
	\[\begin{tikzcd}
		B \arrow[r,"a"] \arrow[d,dashed,"\exists f",swap] & C \arrow[r,"b"] \arrow[d,dashed] & A \arrow[r,dashed,"\delta"] \arrow[d,equal] & {} \\
		\tau A \arrow[r] & D \arrow[r] & A \arrow[r,"\rho",dashed] & {}.
	\end{tikzcd}\]
	Then we have $(\delta,\overline{f})\longmapsto \eta_{A}(\overline{f}_{*}\delta)=\eta_{A}(\rho)=g([\underline{{\rm id}_{A}}])\neq 0$.
	
	Note that the above discussion implies $\mathbb{E}(A,B)\neq 0$ if and only if $\overline{\mathcal{C}}(B,\tau A)\neq 0$. Thus by Ext-finiteness, we have isomorphism of $R$-modules $\phi_{A,B}:\mathbb{E}(A,B)\simeq \mathbb{D}\overline{\mathcal{C}}(B,\tau A)$ given by $\delta \mapsto (\overline{f} \mapsto \eta_{A}(\overline{f}_{*}\delta))$ for any $A,B\in \mathcal{C}$.
	
	It suffices to show naturality of $\phi$. For any $\underline{u}:A'\rightarrow A$ and $\overline{v}:B\rightarrow B'$, consider the diagram
	\[\begin{tikzcd}
		\mathbb{E}(A,B) \arrow[r,"\phi_{A,B}"] \arrow[d,"\underline{u}^{*}\overline{v}_{*}"] & \mathbb{D}\overline{\mathcal{C}}(B,\tau A) \arrow[d,"\mathbb{D}(\tau(\underline{u})_{*}\overline{v}^{*})"] \\
		\mathbb{E}(A',B') \arrow[r,"\phi_{A',B'}"] & \mathbb{D}\overline{\mathcal{C}}(B',\tau A').
	\end{tikzcd}\]
	Let $\delta \in \mathbb{E}(A,B)$ and $\overline{f} \in \overline{\mathcal{C}}(B',\tau A')$, we have
	\[\mathbb{D}(\tau(\underline{u})_{*}\overline{v}^{*})(\phi_{A,B}(\delta))(\overline{f})=\phi_{A,B}(\delta)(\tau(\underline{u})\overline{fv})=\eta_{A}((\tau(\underline{u})\overline{fv})_{*}\delta)=\eta_{A'}(\underline{u}^{*}(\overline{fv})_{*}\delta)=\phi_{A',B'}(\underline{u}^{*}\overline{v}_{*}\delta)(\overline{f}).\]
	This completes the proof.

\end{proof}

\begin{thm}\label{main thm1}
	Let $\mathbb{F}\xhookrightarrow{\iota} \mathbb{E}$ be a closed subfunctor. Suppose $(\mathcal{C},\mathbb{E},\mathfrak{s})$ has a right ARS duality $(\tau,\eta)$. Then it induces a right ARS duality $(\tau_{\mathbb{F}},\eta^{\mathbb{F}})$ on the relative theory $(\mathcal{C},\mathbb{F},\mathfrak{s}|_{\mathbb{F}})$. Moreover, there are commutative diagrams as follows for any $A,B\in \mathcal{C}$.
	\[\begin{tikzcd}
		\underline{\mathcal{C}} \arrow[r,"\tau"] \arrow[d,"\pi"] & \overline{\mathcal{C}} \arrow[d,"\pi'"] \\
		\mathcal{C}/\mathcal{P}_{\mathbb{F}} \arrow[r,"\tau_{\mathbb{F}}"] & \mathcal{C}/\mathcal{I}_{\mathbb{F}}
	\end{tikzcd}\text{ and }
	\begin{tikzcd}
		\underline{\mathcal{C}}(A,B) \arrow[r,"\eta_{A,B}","\sim" swap] \arrow[d,"\pi"] & \mathbb{D}\mathbb{E}(B,\tau A) \arrow[d,"\mathbb{D}(\iota)"] \\
		\mathcal{C}/\mathcal{P}_{\mathbb{F}}(A,B) \arrow[r,"\eta^{\mathbb{F}}_{A,B}","\sim" swap] & \mathbb{D}\mathbb{F}(B,\tau_{\mathbb{F}} A)
	\end{tikzcd}\]
\end{thm}
\begin{proof}
	Since $\mathbb{F}$ is a subfunctor of $\mathbb{E}$, we have $\mathcal{P}_{\mathbb{E}}\subseteq \mathcal{P}_{\mathbb{F}}$ and $\mathcal{I}_{\mathbb{E}}\subseteq \mathcal{I}_{\mathbb{F}}$. Thus there exist canonical functors $\pi:\underline{\mathcal{C}}\rightarrow \mathcal{C}/\mathcal{P}_{\mathbb{F}}$ and $\pi':\overline{\mathcal{C}}\rightarrow \mathcal{C}/\mathcal{I}_{\mathbb{F}}$. Define $\eta_{A}:=\eta_{A,A}(\underline{{\rm id}_{A}})$. Consider the composition
	\[\underline{\mathcal{C}}(A,B)\xrightarrow{\eta_{A,B}} \mathbb{D}\mathbb{E}(B,\tau A)\xrightarrow{\mathbb{D}(\iota)} \mathbb{D}\mathbb{F}(B,\tau A)\text{ for any }A,B\in \mathcal{C}.\]
	Let $f\in \mathcal{P}_{\mathbb{F}}(A,B)$. For any $\delta \in \mathbb{F}(B,\tau A)$, we have $\mathbb{D}(\iota)(\eta_{A,B}(\underline{f}))(\delta)=\eta_{A,B}(\underline{f})(\delta)=\eta_{A}(\underline{f}^{*}\delta)=0$. It implies that the composition $\mathbb{D}(\iota)\circ \eta_{A,B}$ factors through $\pi:\underline{\mathcal{C}}(A,B)\rightarrow \mathcal{C}/\mathcal{P}_{\mathbb{F}}(A,B)$. Thus, we obtain a commutative diagram
	\begin{equation}\label{diagram1}
		\begin{tikzcd}
			\underline{\mathcal{C}}(A,B) \arrow[r,"\eta_{A,B}","\sim" swap] \arrow[d,"\pi"] & \mathbb{D}\mathbb{E}(B,\tau A) \arrow[d,"\mathbb{D}(\iota)"] \\
			\mathcal{C}/\mathcal{P}_{\mathbb{F}}(A,B) \arrow[r] & \mathbb{D}\mathbb{F}(B,\tau A).
		\end{tikzcd}
	\end{equation}
	We show the preimage of $\tau(\underline{f})$ in $\mathcal{C}(\tau A,\tau B)$ lies in $\mathcal{I}_{\mathbb{F}}(\tau A, \tau B)$. Assume otherwise, by definition, there exists $X\in \mathcal{C}$ satisfying $\mathbb{F}(X,\tau(\underline{f}))\neq 0$. Then there exists $\alpha \in \mathbb{F}(X,\tau A)$ such that $\tau(\underline{f})_{*}\alpha \in \mathbb{F}(X,\tau B)$ is nonzero. Since $\mathbb{F}(X,\tau B)\subseteq \mathbb{E}(X,\tau B)\simeq \mathbb{D}\underline{\mathcal{C}}(B,X)$, $\tau(\underline{f})_{*}\alpha$ corresponds to a nonzero element $h\in \mathbb{D}\underline{\mathcal{C}}(B,X)$. Thus there exists $\underline{b}\in \underline{\mathcal{C}}(B,X)$ such that $h(\underline{b})\neq 0$. Consider the following commutative diagram
	\[\begin{tikzcd}
		\mathbb{E}(X,\tau B) \arrow[r,"\underline{b}^{*}"] \arrow[d,"\simeq"] & \mathbb{E}(B,\tau B) \arrow[d,"\simeq"] \\
		\mathbb{D}\underline{\mathcal{C}}(B,X) \arrow[r,"\mathbb{D}(\underline{b}_{*})"] & \mathbb{D}\underline{\mathcal{C}}(B,B).
	\end{tikzcd}\]
	Let $\delta:=\underline{b}^{*}\alpha \in \mathbb{F}(B,\tau A)$. Then we have
	\[\mathbb{D}(\iota)(\eta_{A,B}(\underline{f}))(\delta)=\eta_{A,B}(\underline{f})(\delta)=\eta_{B}(\tau(\underline{f})_{*}\delta)=\eta_{B}(\underline{b}^{*}\tau(\underline{f})_{*}\alpha)=\mathbb{D}(\underline{b}_{*})(h)(\underline{{\rm id}_{B}})=h(\underline{b})\neq 0,\]
	a contradiction. Thus the composition of functors $\underline{\mathcal{C}}\xrightarrow{\tau} \overline{\mathcal{C}}\xrightarrow{\pi'} \mathcal{C}/\mathcal{I}_{\mathbb{F}}$ factors through $\pi:\underline{\mathcal{C}}\rightarrow \mathcal{C}/\mathcal{P}_{\mathbb{F}}$. Denote the induced functor $\mathcal{C}/\mathcal{P}_{\mathbb{F}}\rightarrow \mathcal{C}/\mathcal{I}_{\mathbb{F}}$ by $\tau_{\mathbb{F}}$, then we obtain the first desired commutative diagram
	\begin{equation}\label{diagram2}
		\begin{tikzcd}
			\underline{\mathcal{C}} \arrow[r,"\tau"] \arrow[d,"\pi"] & \overline{\mathcal{C}} \arrow[d,"\pi'"] \\
			\mathcal{C}/\mathcal{P}_{\mathbb{F}} \arrow[r,"\tau_{\mathbb{F}}"] & \mathcal{C}/\mathcal{I}_{\mathbb{F}}.
		\end{tikzcd}
	\end{equation}
	
	Let $\underline{g}\in \underline{\mathcal{C}}(A,B)$ such that $\mathbb{D}(\iota)(\eta_{A,B}(\underline{g}))=0$. We show $g\in \mathcal{P}_{\mathbb{F}}$. Assume otherwise, then there exists $Y\in \mathcal{C}$ satisfying $\mathbb{F}(g,Y)\neq 0$. Thus there exists $\beta \in \mathbb{F}(B,Y)$ such that $\underline{g}^{*}\beta \in \mathbb{F}(A,Y)$ is nonzero. By Lemma \ref{equi def}, there is a binatural isomorphism of $R$-modules $\phi_{A,B}:\mathbb{E}(A,B)\simeq \mathbb{D}\overline{\mathcal{C}}(B,\tau A)$ given by $\delta \mapsto (\overline{f} \mapsto \eta_{A}(\overline{f}_{*}\delta))$ for any $A,B\in \mathcal{C}$. Since $\mathbb{F}(A,Y)\subseteq \mathbb{E}(A,Y)\simeq \mathbb{D}\overline{\mathcal{C}}(Y,\tau A)$, $\underline{g}^{*}\beta$ corresponds to a nonzero element $h'\in \mathbb{D}\overline{\mathcal{C}}(Y,\tau A)$. Thus there exists $\overline{c}\in \overline{\mathcal{C}}(Y,\tau A)$ such that $h'(\overline{c})\neq 0$. Consider the following commutative diagram
	\[\begin{tikzcd}
		\mathbb{E}(A,Y) \arrow[r,"\overline{c}_{*}"] \arrow[d,"\simeq"] & \mathbb{E}(A,\tau A) \arrow[d,"\simeq"] \\
		\mathbb{D}\overline{\mathcal{C}}(Y,\tau A) \arrow[r,"\mathbb{D}(\overline{c}^{*})"] & \mathbb{D}\overline{\mathcal{C}}(\tau A,\tau A).
	\end{tikzcd}\]
	Let $\delta':=\overline{c}_{*}\beta \in \mathbb{F}(B,\tau A)$. Then we have
	\[\mathbb{D}(\iota)(\eta_{A,B}(\underline{g}))(\delta')=\eta_{A,B}(\underline{g})(\delta')=\eta_{A}(\underline{g}^{*}\delta')=\eta_{A}(\overline{c}_{*}\underline{g}^{*}\beta)=\mathbb{D}(\overline{c}^{*})(h')(\overline{{\rm id}_{\tau A}})=h'(\overline{c})\neq 0\]
	a contradiction. This implies that the $R$-module homomorphism in the second row of (\ref{diagram1}) is an isomorphism. Thus we obtain $\eta^{\mathbb{F}}_{A,B}:\mathcal{C}/\mathcal{P}_{\mathbb{F}}(A,B)\simeq \mathbb{D}\mathbb{F}(B,\tau_{\mathbb{F}} A)$ by (\ref{diagram2}), and the second desired commutative diagram. Moreover, the naturality of $\eta^{\mathbb{F}}_{A,B}$ follows immediately from that of $\pi,\eta,\iota$ and (\ref{diagram1}). This completes the proof.
\end{proof}

Next we show that right ARS dualities are unique under Ext-finiteness.

\begin{prop}\label{uniqueness of right ARS}
	Suppose $(\tau,\eta)$ and $(\tau',\eta')$ are both right ARS dualities on $(\mathcal{C},\mathbb{E},\mathfrak{s})$. Then there exists a unique natural isomorphism $\alpha:\tau\xRightarrow{\sim} \tau'$ satisfying the following commutative diagram.
	\[\begin{tikzcd}
		\underline{\mathcal{C}}(A,B) \arrow[r,"\eta^{'}_{A{,}B}","\sim" swap] \arrow[d,equal] & \mathbb{D}\mathbb{E}(B,\tau'A) \arrow[d,"\mathbb{D}((\alpha_{A})_{*})"] \\
		\underline{\mathcal{C}}(A,B) \arrow[r,"\eta_{A{,}B}","\sim" swap] & \mathbb{D}\mathbb{E}(B,\tau A)
	\end{tikzcd}\]
\end{prop}
\begin{proof}
	By Lemma \ref{equi def}, we have $\overline{\mathcal{C}}(-,\tau A)\simeq \mathbb{D}\mathbb{E}(A,-)\simeq \overline{\mathcal{C}}(-,\tau'A)$. By Yoneda's lemma, the composition corresponds to a functorial isomorphism $\alpha_{A}:\tau A\rightarrow \tau'A$ in $\overline{\mathcal{C}}$. Let $\eta_{A}:=\eta_{A,A}(\underline{{\rm id}_{A}})$ and $\eta'_{A}:=\eta'_{A,A}(\underline{{\rm id}_{A}})$. By definition of $\alpha$, we have $\eta_{A}(\delta)=\eta'_{A}((\alpha_{A})_{*}\delta)$ for any $\delta \in \mathbb{E}(A,\tau A)$. Then for any $\underline{f}\in \underline{\mathcal{C}}(A,B)$ and $\rho \in \mathbb{E}(B,\tau A)$, we have
	\[\mathbb{D}((\alpha_{A})_{*})( \eta'_{A,B}(\underline{f}))(\rho)=\eta'_{A,B}(\underline{f})((\alpha_{A})_{*}\rho)=\eta'_{A}(\underline{f}^{*}(\alpha_{A})_{*}\rho)=\eta_{A}(\underline{f}^{*}\rho)=\eta_{A,B}(\underline{f})(\rho),\]
	which proves the commutativity. Let $\widetilde{\alpha}:\tau\xRightarrow{\sim} \tau'$ be another natural isomorphism which makes the diagram commute. Then $\mathbb{E}(-,\alpha_{A}-\widetilde{\alpha}_{A})$ is zero. By definition, we have $\alpha_{A}=\widetilde{\alpha}_{A}$.
\end{proof}

\begin{prop}\label{left & right ARS duality}
	Suppose $(\mathcal{C},\mathbb{E},\mathfrak{s})$ has a right ARS duality $(\tau,\eta)$. The following are equivalent.
	\begin{enumerate}
		\item $(\tau,\eta)$ is an ARS duality.
		\item There exists a left ARS duality on $(\mathcal{C},\mathbb{E},\mathfrak{s})$.
	\end{enumerate}
\end{prop}
\begin{proof}
	$(1)\Rightarrow (2)$ is mentioned in \cite[Lemma 3.9]{INP}. Indeed, a quasi-inverse $\sigma$ of $\tau$ gives a left ARS duality $\overline{\mathcal{C}}(A,B)\simeq \underline{\mathcal{C}}(\sigma A,\sigma B)\simeq \mathbb{D}\mathbb{E}(\sigma B,\tau \sigma A)\simeq \mathbb{D}\mathbb{E}(\sigma B,A)$. For $(2)\Rightarrow (1)$, suppose $(\sigma,\zeta)$ is a left ARS duality on $(\mathcal{C},\mathbb{E},\mathfrak{s})$. That is, an $R$-linear functor $\sigma: \overline{\mathcal{C}}\rightarrow \underline{\mathcal{C}}$ and a binatural isomorphism of $R$-modules $\zeta_{A,B}:\overline{\mathcal{C}}(A,B)\simeq \mathbb{D}\mathbb{E}(\sigma B,A)$ for any $A,B\in \mathcal{C}$. We show that $\sigma$ is a quasi-inverse of $\tau$. By Lemma \ref{equi def} and Ext-finiteness, we have $\overline{\mathcal{C}}(-,B)\simeq \mathbb{D}\mathbb{E}(\sigma B,-)\simeq \overline{\mathcal{C}}(-,\tau \sigma B)$. By Yoneda's lemma, the composition corresponds to a functorial isomorphism $B\simeq \tau \sigma B$ in $\overline{\mathcal{C}}$. Thus we obtain an natural isomorphism ${\rm id}_{\overline{\mathcal{C}}}\simeq \tau \sigma$. Similarly, we obtain ${\rm id}_{\underline{\mathcal{C}}}\simeq \sigma \tau$. Thus (1) follows.
\end{proof}

\section{The exactness of $\tau$}

In this section, we show that the functor $\tau$ in a right ARS duality $(\tau,\eta)$ is indeed an exact functor under relative structures. As in the previous section, let $R$ be a commutative artin ring and $(\mathcal{C},\mathbb{E},\mathfrak{s})$ be an $R$-linear Ext-finite extriangulated category. We need the following condition \cite[Condition 5.8]{NP}.

\begin{enumerate}
	\item [(\bf WIC)] For any morphisms $f\in \mathcal{C}(X,Y)$ and $g\in \mathcal{C}(Y,Z)$, if $g\circ f$ is an $\mathfrak{s}$-inflation, then so is $f$. Dually, if $g\circ f$ is an $\mathfrak{s}$-deflation, then so is $g$.
\end{enumerate}

It is interesting that this condition does not depend on the extriangulated structure on $\mathcal{C}$, but the additive structure on $\mathcal{C}$, which is the same as that of an exact category \cite[Proposition 7.6]{Buhler}.

\begin{lem}\cite[Proposition 2.7]{K}\label{WIC=w.i.c.}
	An extriangulated category satisfies condition (WIC) if and only if it is weakly idempotent complete.
\end{lem}

\begin{thm}\label{main thm2}
	Suppose $(\mathcal{C},\mathbb{E},\mathfrak{s})$ is weakly idempotent complete, has a right ARS duality $(\tau,\eta)$ and enough projectives $\mathcal{Q}$ and enough injectives $\mathcal{J}$ which are both functorially finite. Define
	\[(\underline{\mathcal{C}},\underline{\mathbb{E}^{\mathcal{Q}}},\underline{\mathfrak{s}|_{\mathbb{E}^{\mathcal{Q}}}})\,(\text{resp. }(\overline{\mathcal{C}},\overline{\mathbb{E}_{\mathcal{J}}},\overline{\mathfrak{s}|_{\mathbb{E}_{\mathcal{J}}}}))\]
	to be the ideal quotient of $(\mathcal{C},\mathbb{E}^{\mathcal{Q}},\mathfrak{s}|_{\mathbb{E}^{\mathcal{Q}}})\,(\text{resp. } (\mathcal{C},\mathbb{E}_{\mathcal{J}},\mathfrak{s}|_{\mathbb{E}_{\mathcal{J}}}))$ by $\mathcal{Q}\,(\text{resp. } \mathcal{J})$. Then there exists a natural isomorphism $\phi:\underline{\mathbb{E}^{\mathcal{Q}}}\xRightarrow{\sim} \overline{\mathbb{E}_{\mathcal{J}}}\circ (\tau^{\rm op}\times \tau)$ such that
	\[(\tau,\phi):(\underline{\mathcal{C}},\underline{\mathbb{E}^{\mathcal{Q}}},\underline{\mathfrak{s}|_{\mathbb{E}^{\mathcal{Q}}}})\rightarrow (\overline{\mathcal{C}},\overline{\mathbb{E}_{\mathcal{J}}},\overline{\mathfrak{s}|_{\mathbb{E}_{\mathcal{J}}}})\]
	is an exact functor which is fully faithful.
\end{thm}
\begin{proof}
	By definition, we have $\mathcal{P}_{\mathbb{E}}=[\mathcal{Q}]$ (resp. $\mathcal{I}_{\mathbb{E}}=[\mathcal{J}]$) and therefore $\underline{\mathcal{C}}=\mathcal{C}/[\mathcal{Q}]$ (resp. $\overline{\mathcal{C}}=\mathcal{C}/[\mathcal{J}]$). By Example \ref{typical exam of subfunctor}, objects in $\mathcal{Q}$ (resp. $\mathcal{J}$) are projective-injective in the relative theory $(\mathcal{C},\mathbb{E}^{\mathcal{Q}},\mathfrak{s}|_{\mathbb{E}^{\mathcal{Q}}})$ (resp. $(\mathcal{C},\mathbb{E}_{\mathcal{J}},\mathfrak{s}|_{\mathbb{E}_{\mathcal{J}}})$). Thus the ideal quotient $(\underline{\mathcal{C}},\underline{\mathbb{E}^{\mathcal{Q}}},\underline{\mathfrak{s}|_{\mathbb{E}^{\mathcal{Q}}}})$ and $(\overline{\mathcal{C}},\overline{\mathbb{E}_{\mathcal{J}}},\overline{\mathfrak{s}|_{\mathbb{E}_{\mathcal{J}}}})$ are extriangulated categories by Lemma \ref{induced extri str on ideal quotient}. Let $\eta_{A}:=\eta_{A,A}(\underline{{\rm id}_{A}})$ be as usual.
	
	(1) We firstly define a natural transformation $\phi:\underline{\mathbb{E}^{\mathcal{Q}}}\Rightarrow \overline{\mathbb{E}_{\mathcal{J}}}\circ (\tau^{\rm op}\times \tau)$. By Theorem \ref{main thm1}, right ARS duality $(\tau,\eta)$ on $(\mathcal{C},\mathbb{E},\mathfrak{s})$ induces right ARS dualities on $(\mathcal{C},\mathbb{E}^{\mathcal{Q}},\mathfrak{s}|_{\mathbb{E}^{\mathcal{Q}}})$ and $(\mathcal{C},\mathbb{E}_{\mathcal{J}},\mathfrak{s}|_{\mathbb{E}_{\mathcal{J}}})$. Indeed, we have
	\[\tau_{\mathbb{E}^{\mathcal{Q}}}:\mathcal{C}/\mathcal{P}_{\mathbb{E}^{\mathcal{Q}}}\rightarrow \mathcal{C}/\mathcal{I}_{\mathbb{E}^{\mathcal{Q}}},~\eta^{\mathbb{E}^{\mathcal{Q}}}_{A,B}:\mathcal{C}/\mathcal{P}_{\mathbb{E}^{\mathcal{Q}}}(A,B)\simeq \mathbb{D}\mathbb{E}^{\mathcal{Q}}(B,\tau_{\mathbb{E}^{\mathcal{Q}}}A)\]
	and
	\[\tau_{\mathbb{E}_{\mathcal{J}}}:\mathcal{C}/\mathcal{P}_{\mathbb{E}_{\mathcal{J}}}\rightarrow \mathcal{C}/\mathcal{I}_{\mathbb{E}_{\mathcal{J}}},~\eta^{\mathbb{E}_{\mathcal{J}}}_{A,B}:\mathcal{C}/\mathcal{P}_{\mathbb{E}_{\mathcal{J}}}(A,B)\simeq \mathbb{D}\mathbb{E}_{\mathcal{J}}(B,\tau_{\mathbb{E}_{\mathcal{J}}}A)\]
	for any $A,B\in \mathcal{C}$. We claim that $\mathcal{I}_{\mathbb{E}^{\mathcal{Q}}}=\mathcal{P}_{\mathbb{E}_{\mathcal{J}}}=[{\rm add}(\mathcal{Q}\cup \mathcal{J})]$. We only show $\mathcal{I}_{\mathbb{E}^{\mathcal{Q}}}=[{\rm add}(\mathcal{Q}\cup \mathcal{J})]$, the other one follows dually. For any $X\in \mathcal{C}$, there exists an $\mathfrak{s}$-triangle $X\rightarrow I_{X}\rightarrow Y\dashrightarrow $ satisfying $I_{X}\in \mathcal{J}$. Since $\mathcal{Q}$ is covariantly finite, there exists a left $\mathcal{Q}$-approximation $X\rightarrow Q$ of $X$. By Lemma \ref{weak pushout}, we obtain an $\mathfrak{s}$-triangle $X\rightarrow Q\oplus I_{X}\rightarrow X'\dashrightarrow$. It is an $\mathfrak{s}|_{\mathbb{E}^{\mathcal{Q}}}$-triangle and $Q\oplus I_{X}$ is injective in $(\mathcal{C},\mathbb{E}^{\mathcal{Q}},\mathfrak{s}|_{\mathbb{E}^{\mathcal{Q}}})$. Then $(\mathcal{C},\mathbb{E}^{\mathcal{Q}},\mathfrak{s}|_{\mathbb{E}^{\mathcal{Q}}})$ has enough injectives ${\rm add}(\mathcal{Q}\cup \mathcal{J})$ and therefore $\mathcal{I}_{\mathbb{E}^{\mathcal{Q}}}=[{\rm add}(\mathcal{Q}\cup \mathcal{J})]$. Thus we have
	\begin{align*}
		\underline{\mathbb{E}^{\mathcal{Q}}}(C,A)&=\mathbb{E}^{\mathcal{Q}}(C,A) &&\text{by Definition}\\
		&\simeq \mathbb{D}(\mathcal{C}/\mathcal{I}_{\mathbb{E}^{\mathcal{Q}}})(A,\tau_{\mathbb{E}^{\mathcal{Q}}}C) &&\text{by Lemma \ref{equi def}}\\
		&=\mathbb{D}(\mathcal{C}/\mathcal{P}_{\mathbb{E}_{\mathcal{J}}})(A,\tau_{\mathbb{E}_{\mathcal{J}}}C) &&\text{by }\mathcal{I}_{\mathbb{E}^{\mathcal{Q}}}=\mathcal{P}_{\mathbb{E}_{\mathcal{J}}} \text{ and Theorem \ref{main thm1}}\\
		&\simeq \mathbb{E}_{\mathcal{J}}(\tau_{\mathbb{E}_{\mathcal{J}}}C,\tau_{\mathbb{E}_{\mathcal{J}}}A) &&\text{by Ext-finiteness}\\
		&=\overline{\mathbb{E}_{\mathcal{J}}}(\tau C,\tau A) &&\text{by Theorem \ref{main thm1}}
	\end{align*}
	We denote this isomorphism by $\phi_{C,A}$. Clearly, it is a natural isomorphism.
	
	(2) Let $\delta \in \underline{\mathbb{E}^{\mathcal{Q}}}(C,A)=\mathbb{E}^{\mathcal{Q}}(C,A)\subseteq \mathbb{E}(C,A)$, take an $\mathfrak{s}|_{\mathbb{E}^{\mathcal{Q}}}$-triangle $A\xrightarrow{a} B\xrightarrow{b} C\stackrel{\delta}\dashrightarrow $. Then $A\xrightarrow{\underline{a}} B\xrightarrow{\underline{b}} C\stackrel{\delta}\dashrightarrow $ is an $\underline{\mathfrak{s}|_{\mathbb{E}^{\mathcal{Q}}}}$-triangle in $(\underline{\mathcal{C}},\underline{\mathbb{E}^{\mathcal{Q}}},\underline{\mathfrak{s}|_{\mathbb{E}^{\mathcal{Q}}}})$. It suffices to show
	\begin{equation}\label{triangle}
		\begin{tikzcd}
			\tau A \arrow[r,"\tau(\underline{a})"] & \tau B \arrow[r,"\tau(\underline{b})"] & \tau C \arrow[r,"\phi_{C,A}(\delta)",dashed] & {}
		\end{tikzcd}
	\end{equation}
	is an $\overline{\mathfrak{s}|_{\mathbb{E}_{\mathcal{J}}}}$-triangle in $(\overline{\mathcal{C}},\overline{\mathbb{E}_{\mathcal{J}}},\overline{\mathfrak{s}|_{\mathbb{E}_{\mathcal{J}}}})$. But firstly, we need the following claim.
	
	\begin{claim}\label{cliam}
		There are two exact sequences
		\begin{enumerate}
			\item [(i)] $\overline{\mathcal{C}}(-,\tau A)\xrightarrow{\overline{\mathcal{C}}(-,\tau(\underline{a}))} \overline{\mathcal{C}}(-,\tau B)\xrightarrow{\overline{\mathcal{C}}(-,\tau(\underline{b}))} \overline{\mathcal{C}}(-,\tau C)\xrightarrow{{\phi_{C,A}(\delta)}_{\sharp(-)}} \overline{\mathbb{E}_{\mathcal{J}}}(-,\tau A)\\
			\xrightarrow{\tau(\underline{a})_{*(-)}} \overline{\mathbb{E}_{\mathcal{J}}}(-,\tau B) \xrightarrow{\tau(\underline{b})_{*(-)}} \overline{\mathbb{E}_{\mathcal{J}}}(-,\tau C) \text{ is exact in } {\rm Mod}\,\overline{\mathcal{C}}.$
			\item [(ii)] $\overline{\mathcal{C}}(\tau C,\tau(-)) \xrightarrow{\overline{\mathcal{C}}(\tau(\underline{b}),\tau(-))} \overline{\mathcal{C}}(\tau B,\tau(-))\xrightarrow{\overline{\mathcal{C}}(\tau(\underline{a}),\tau(-))} \overline{\mathcal{C}}(\tau A,\tau(-))\xrightarrow{{\phi_{C,A}(\delta)}^{\sharp}_{\tau(-)}} \overline{\mathbb{E}_{\mathcal{J}}}(\tau C,\tau(-))\\
			\xrightarrow{\tau(\underline{b})^{*}_{\tau(-)}} \overline{\mathbb{E}_{\mathcal{J}}}(\tau B,\tau(-))\xrightarrow{\tau(\underline{a})^{*}_{\tau(-)}} \overline{\mathbb{E}_{\mathcal{J}}}(\tau A,\tau(-)) \text{ is exact in } {\rm Mod}\,\underline{\mathcal{C}}^{\rm op}$.
		\end{enumerate}
	\end{claim}
	\begin{proof}[Proof of Claim\ref{cliam}]
		(i) By Ext-finiteness of $\mathcal{C}$ and functorial isomorphisms $\overline{\mathcal{C}}(-,\tau A)\simeq \mathbb{D}\mathbb{E}(A,-)$ and $\overline{\mathbb{E}_{\mathcal{J}}}(-,\tau A)\simeq \mathbb{D}(\mathcal{C}/\mathcal{P}_{\mathbb{E}_{\mathcal{J}}})(A,-)$, it suffices to show the exactness of the following sequence
		\begin{equation}\label{exact seq3}
			\begin{split}
				\mathcal{C}/\mathcal{P}_{\mathbb{E}_{\mathcal{J}}}(C,-) & \xrightarrow{\mathcal{C}/\mathcal{P}_{\mathbb{E}_{\mathcal{J}}}([b],-)} \mathcal{C}/\mathcal{P}_{\mathbb{E}_{\mathcal{J}}}(B,-)\xrightarrow{\mathcal{C}/\mathcal{P}_{\mathbb{E}_{\mathcal{J}}}([a],-)} \mathcal{C}/\mathcal{P}_{\mathbb{E}_{\mathcal{J}}}(A,-)\xrightarrow{\delta^{\sharp}_{(-)}} \mathbb{E}(C,-)\\
				& \xrightarrow{\underline{b}^{*}_{(-)}} \mathbb{E}(B,-)\xrightarrow{\underline{a}^{*}_{(-)}} \mathbb{E}(A,-)
			\end{split}
		\end{equation}
		and the commutative diagram
		\[\begin{tikzcd}
			\overline{\mathcal{C}}(-,\tau C) \arrow[r,"{\phi_{C,A}(\delta)}_{\sharp(-)}"] \arrow[d,"\simeq"] & \overline{\mathbb{E}_{\mathcal{J}}}(-,\tau A) \arrow[d,"\simeq"] \\
			\mathbb{D}\mathbb{E}(C,-) \arrow[r,"\mathbb{D}(\delta^{\sharp}_{(-)})"] & \mathbb{D}(\mathcal{C}/\mathcal{P}_{\mathbb{E}_{\mathcal{J}}})(A,-).
		\end{tikzcd}\]
		By Lemma \ref{exact seq} and \ref{(co)stable cat}, there is an exact sequence
		\[\mathcal{C}(C,-)\xrightarrow{\mathcal{C}(b,-)} \mathcal{C}(B,-)\xrightarrow{\mathcal{C}(a,-)} \mathcal{C}(A,-)\xrightarrow{\delta^{\sharp}_{(-)}} \mathbb{E}(C,-)\xrightarrow{\underline{b}^{*}_{(-)}} \mathbb{E}(B,-)\xrightarrow{\underline{a}^{*}_{(-)}} \mathbb{E}(A,-).\]
		Since $\mathcal{I}_{\mathbb{E}^{\mathcal{Q}}}=\mathcal{P}_{\mathbb{E}_{\mathcal{J}}}=[{\rm add}(\mathcal{Q}\cup \mathcal{J})]$ and $\delta \in \mathbb{E}^{\mathcal{Q}}(C,A)$, by the dual of \cite[Lemma 1.26]{INP}, the exactness of (\ref{exact seq3}) follows. To prove commutativity, by Yoneda's lemma, it suffices to show $\eta_{C}([f]_{*}\delta)=\eta_{A}([f]^{*}\phi_{C,A}(\delta))$ for any $[f]\in \mathcal{C}/\mathcal{P}_{\mathbb{E}_{\mathcal{J}}}(A,\tau C)$. This follows immediately from the definition of $\phi_{C,A}$.
		
		(ii) Since $A\xrightarrow{\underline{a}} B\xrightarrow{\underline{b}} C\stackrel{\delta}\dashrightarrow $ is an $\underline{\mathfrak{s}|_{\mathbb{E}^{\mathcal{Q}}}}$-triangle in $(\underline{\mathcal{C}},\underline{\mathbb{E}^{\mathcal{Q}}},\underline{\mathfrak{s}|_{\mathbb{E}^{\mathcal{Q}}}})$, by Lemma \ref{exact seq}, we obtain an exact sequence
		\[\underline{\mathcal{C}}(C,-)\xrightarrow{\underline{\mathcal{C}}(\underline{b},-)} \underline{\mathcal{C}}(B,-)\xrightarrow{\underline{\mathcal{C}}(\underline{a},-)} \underline{\mathcal{C}}(A,-)\xrightarrow{\delta^{\sharp}} \underline{\mathbb{E}^{\mathcal{Q}}}(C,-)\xrightarrow{\underline{b}^{*}} \underline{\mathbb{E}^{\mathcal{Q}}}(B,-)\xrightarrow{\underline{a}^{*}} \underline{\mathbb{E}^{\mathcal{Q}}}(A,-).\]
		By fully faithfulness of $\tau$ and natural isomorphism $\underline{\mathbb{E}^{\mathcal{Q}}}(C,A)\simeq \overline{\mathbb{E}_{\mathcal{J}}}(\tau C,\tau A)$, we obtain the required exact sequence in ${\rm Mod}\,\underline{\mathcal{C}}^{\rm op}$.
	\end{proof}
	(3) Suppose $\phi_{C,A}(\delta)$ is realized by $\tau A\xrightarrow{\overline{a'}} B'\xrightarrow{\overline{b'}} \tau C$ in $(\overline{\mathcal{C}},\overline{\mathbb{E}_{\mathcal{J}}},\overline{\mathfrak{s}|_{\mathbb{E}_{\mathcal{J}}}})$, then $\overline{b'}^{*}\phi_{C,A}(\delta)=0$ by Lemma \ref{exact seq}. By Claim \ref{cliam} (i), there exists $\overline{c'}\in \overline{\mathcal{C}}(B',\tau B)$ satisfying $\tau(\underline{b})\circ \overline{c'}=\overline{b'}$. Since $\mathcal{C}$ is weakly idempotent complete, so is $\overline{\mathcal{C}}$. By Lemma \ref{WIC=w.i.c.}, $\tau(\underline{b})$ is an $\overline{\mathfrak{s}|_{\mathbb{E}_{\mathcal{J}}}}$-deflation. Suppose there is an $\overline{\mathfrak{s}|_{\mathbb{E}_{\mathcal{J}}}}$-triangle $A'\xrightarrow{\overline{\alpha}} \tau B\xrightarrow{\tau(\underline{b})} \tau C\stackrel{\beta}\dashrightarrow $ with $\overline{\alpha}\in \overline{\mathcal{C}}(A',\tau B) $ and $\beta \in \overline{\mathbb{E}_{\mathcal{J}}}(\tau C,A')$. In this part, we show that if there exists $\overline{h}\in \overline{\mathcal{C}}(A',\tau A)$ such that $\overline{h}_{*}\beta=\phi_{C,A}(\delta)$ and the following diagram commutes, then the proof is complete.
	\[\begin{tikzcd}
		A' \arrow[r,"\overline{\alpha}"] \arrow[d,"\overline{h}",dashed] & \tau B \arrow[r,"\tau(\underline{b})"] \arrow[d,equal] & \tau C \arrow[r,"\beta",dashed] \arrow[d,equal] & {} \\
		\tau A \arrow[r,"\tau(\underline{a})"] & \tau B \arrow[r,"\tau(\underline{b})"] & \tau C & {}
	\end{tikzcd}\]
	Indeed, by Claim \ref{cliam} (ii), we obtain the following commutative diagram in ${\rm Mod}\,\underline{\mathcal{C}}^{\rm op}$
	\[\begin{tikzcd}
		\overline{\mathcal{C}}(\tau C,\tau(-)) \arrow[r] \arrow[d,equal] & \overline{\mathcal{C}}(\tau B,\tau(-)) \arrow[r] \arrow[d,equal] & \overline{\mathcal{C}}(\tau A,\tau(-)) \arrow[r] \arrow[d] & \overline{\mathbb{E}_{\mathcal{J}}}(\tau C,\tau(-)) \arrow[r] \arrow[d,equal] & \overline{\mathbb{E}_{\mathcal{J}}}(\tau B,\tau(-)) \arrow[d,equal] \\
		\overline{\mathcal{C}}(\tau C,\tau(-)) \arrow[r] & \overline{\mathcal{C}}(\tau B,\tau(-)) \arrow[r] & \overline{\mathcal{C}}(A',\tau(-)) \arrow[r] & \overline{\mathbb{E}_{\mathcal{J}}}(\tau C,\tau(-)) \arrow[r] & \overline{\mathbb{E}_{\mathcal{J}}}(\tau B,\tau(-)).
	\end{tikzcd}\]
	The third square commutes because of the equality $\overline{h}_{*}\beta=\phi_{C,A}(\delta)$. By the five lemma, we obtain isomorphism $\overline{\mathcal{C}}(\tau A,\tau(-))\simeq \overline{\mathcal{C}}(A',\tau(-))$. Then we have $\mathbb{E}(-,\tau A)\simeq \mathbb{E}(-,A')$. By \cite[Theorem 4.1]{INP}, the morphism $\overline{h}: A'\rightarrow \tau A$ is an isomorphism in $\overline{\mathcal{C}}$. Thus by \cite[Proposition 3.7]{NP}, (\ref{triangle}) is an $\overline{\mathfrak{s}|_{\mathbb{E}_{\mathcal{J}}}}$-triangle in $(\overline{\mathcal{C}},\overline{\mathbb{E}_{\mathcal{J}}},\overline{\mathfrak{s}|_{\mathbb{E}_{\mathcal{J}}}})$.
	
	(4) Suppose there exists $\overline{h}\in \overline{\mathcal{C}}(A',\tau A)$ satisfying $\tau(\underline{a})\circ \overline{h}=\overline{\alpha}$. Since $\overline{\mathcal{C}}(A',\tau B)\simeq \mathbb{D}\mathbb{E}(B,A')$, then for any $\rho\in \mathbb{E}(B,A')$, we have $\eta_{B}((\tau(\underline{a})\circ \overline{h})_{*}\rho)=\eta_{B}(\overline{\alpha}_{*}\rho)$. We also have $\eta_{B}((\tau(\underline{a})\circ \overline{h})_{*}\rho)=\eta_{B}(\tau(\underline{a})_{*}\overline{h}_{*}\rho)=\eta_{A}(\underline{a}^{*}\overline{h}_{*}\rho)$. Thus we obtain
	\begin{equation}\label{equi condition1}
		\eta_{A}(\underline{a}^{*}\overline{h}_{*}\rho)=\eta_{B}(\overline{\alpha}_{*}\rho) \text{ for any } \rho \in \mathbb{E}(B,A').
	\end{equation}
	Conversely, if (\ref{equi condition1}) holds, then clearly $\tau(\underline{a})\circ \overline{h}=\overline{\alpha}$.
	
	Suppose there exists $\overline{h}\in \overline{\mathcal{C}}(A',\tau A)$ satisfying $\overline{h}_{*}\beta=\phi_{C,A}(\delta)\in \overline{\mathbb{E}_{\mathcal{J}}}(\tau C,\tau A)$. Since $\overline{\mathbb{E}_{\mathcal{J}}}(\tau C,\tau A)\simeq \mathbb{D}(\mathcal{C}/\mathcal{P}_{\mathbb{E}_{\mathcal{J}}})(A,\tau C)$, then for any $[f]\in \mathcal{C}/\mathcal{P}_{\mathbb{E}_{\mathcal{J}}}(A,\tau C)$, we have $\eta_{A}([f]^{*}\overline{h}_{*}\beta)=\eta_{A}([f]^{*}\phi_{C,A}(\delta))$. By the proof of Claim \ref{cliam} (i), we also have $\eta_{A}([f]^{*}\phi_{C,A}(\delta))=\eta_{C}([f]_{*}\delta)$. Thus we obtain
	\begin{equation}\label{equi condition2}
		\eta_{A}([f]^{*}\overline{h}_{*}\beta)=\eta_{C}([f]_{*}\delta) \text{ for any } [f]\in \mathcal{C}/\mathcal{P}_{\mathbb{E}_{\mathcal{J}}}(A,\tau C).
	\end{equation}
	Conversely, if (\ref{equi condition2}) holds, then clearly $\overline{h}_{*}\beta=\phi_{C,A}(\delta)$.
	
	Thus, there exists $\overline{h}\in \overline{\mathcal{C}}(A',\tau A)$ satisfying $\tau(\underline{a})\circ \overline{h}=\overline{\alpha}$ and $\overline{h}_{*}\beta=\phi_{C,A}(\delta)$ if and only if there exists $\overline{h}\in \overline{\mathcal{C}}(A',\tau A)$ such that (\ref{equi condition1}) and (\ref{equi condition2}) hold.
		
	(5) In this part, we show that there exists $\overline{h}\in \overline{\mathcal{C}}(A',\tau A)$ such that (\ref{equi condition1}) and (\ref{equi condition2}) hold is equivalent to the following 
	\begin{equation}\label{equi condition}
		\text{If }\rho \in \mathbb{E}(B,A') \text{ and }[f]\in \mathcal{C}/\mathcal{P}_{\mathbb{E}_{\mathcal{J}}}(A,\tau C)\text{ satisfy }\underline{a}^{*}\rho=[f]^{*}\beta \text{, then } \eta_{B}(\overline{\alpha}_{*}\rho)=\eta_{C}([f]_{*}\delta).
	\end{equation}
	It suffices to show that if (\ref{equi condition}) holds, then there exists $\overline{h}\in \overline{\mathcal{C}}(A',\tau A)$ such that (\ref{equi condition1}) and (\ref{equi condition2}) hold, since the other direction is obvious. Consider the following diagram in ${\rm mod}R$
	\[\begin{tikzcd}
		0 \arrow[r] & {\rm Ker}u \arrow[r,"i"] & \mathbb{E}(B,A')\oplus \mathcal{C}/\mathcal{P}_{\mathbb{E}_{\mathcal{J}}}(A,\tau C) \arrow[r,"u"] \arrow[d,"v"] & \mathbb{E}(A,A') \arrow[dl,"\exists g" swap, bend left, dashed] \\
		& & J & 
	\end{tikzcd}\]
	in which the first row is exact and $u=\begin{pmatrix}
	\underline{a}^{*} & -\beta_{\sharp}
	\end{pmatrix}, v=\begin{pmatrix}
	\eta_{B}\circ \overline{\alpha}_{*} & -\eta_{C}\circ \delta^{\sharp}
	\end{pmatrix}$. 
	Then $(\rho,[f])\in {\rm Ker}u$ if and only if $\underline{a}^{*}\rho=[f]^{*}\beta$. Thus (\ref{equi condition}) holds if and only if $vi=0$. Since $J$ is injective in ${\rm mod}R$, there exists $g:\mathbb{E}(A,A')\rightarrow J$ such that $g\circ u=v$. Then $g\in \mathbb{D}\mathbb{E}(A,A')\simeq \overline{\mathcal{C}}(A',\tau A)$. Suppose $g$ corresponds to $\overline{h}\in \overline{\mathcal{C}}(A',\tau A)$, then clearly $g\circ u=v$ is equivalent to that both (\ref{equi condition1}) and (\ref{equi condition2}) hold.
	
	(6) In this final part, we show (\ref{equi condition}) holds. Suppose $\rho \in \mathbb{E}(B,A')$ is realized by $A'\rightarrow D\rightarrow B$. By ${\rm (ET4)^{op}}$, we obtain the following commutative diagram in $\mathcal{C}$.
	\[\begin{tikzcd}
		A' \arrow[r,"d"] \arrow[d,equal] & E \arrow[r,"e"] \arrow[d] & A \arrow[r,dashed,"a^{*}\rho"] \arrow[d,"a"] & {} \\
		A' \arrow[r] & D \arrow[r] \arrow[d] & B \arrow[r,dashed,"\rho"] \arrow[d,"b"] & {} \\
		& C \arrow[r,equal] \arrow[d,dashed,"\sigma"] & C \arrow[d,dashed,"\delta"] & \\
		& {} & {} & 
	\end{tikzcd}\]
	Since $\underline{a}^{*}\rho=[f]^{*}\beta \in \overline{\mathbb{E}_{\mathcal{J}}}(A,A')$, we obtain a commutative diagram in $\overline{\mathcal{C}}$ as follows.
	\[\begin{tikzcd}
		A' \arrow[r,"\overline{d}"] \arrow[d,equal] & E \arrow[r,"\overline{e}"] \arrow[d,"\overline{c}"] & A \arrow[r,dashed,"a^{*}\rho"] \arrow[d,"\overline{f}"] & {} \\
		A' \arrow[r,"\overline{\alpha}"] & \tau B \arrow[r,"\tau(\underline{b})"] & \tau C \arrow[r,dashed,"\beta"] & {}
	\end{tikzcd}\]
	Then $\eta_{B}(\overline{\alpha}_{*}\rho)=\eta_{B}(\overline{c}_{*}\overline{d}_{*}\rho)=\eta_{B}(\overline{c}_{*}\underline{b}^{*}\sigma)=\eta_{C}(\tau(\underline{b})_{*}\overline{c}_{*}\sigma)=\eta_{C}(\overline{f}_{*}\overline{e}_{*}\sigma)=\eta_{C}(\overline{f}_{*}\delta)=\eta_{C}([f]_{*}\delta)$. The proof is now complete.
\end{proof}

\begin{cor}
	Suppose $(\mathcal{C},\mathbb{E},\mathfrak{s})$ satisfies the setting of Theorem \ref{main thm2}. Let $(\tau,\eta),(\tau',\eta')$ be right ARS dualities on $(\mathcal{C},\mathbb{E},\mathfrak{s})$ and $(\tau,\phi),(\tau',\phi'):(\underline{\mathcal{C}},\underline{\mathbb{E}^{\mathcal{Q}}},\underline{\mathfrak{s}|_{\mathbb{E}^{\mathcal{Q}}}})\rightarrow (\overline{\mathcal{C}},\overline{\mathbb{E}_{\mathcal{J}}},\overline{\mathfrak{s}|_{\mathbb{E}_{\mathcal{J}}}})$ be the corresponding exact functors obtained in Theorem \ref{main thm2}, and  $\alpha:\tau\xRightarrow{\sim} \tau'$ be the natural isomorphism obtained in Proposition \ref{uniqueness of right ARS}. Then $\alpha$ is a natural isomorphism of exact functors.
\end{cor}
\begin{proof}
	Let $\delta \in \underline{\mathbb{E}^{\mathcal{Q}}}(C,A)$, we need to show $(\alpha_{A})_{*}\phi_{C,A}(\delta)=(\alpha_{C})^{*}\phi'_{C,A}(\delta)$. Let $\eta_{A}:=\eta_{A,A}(\underline{{\rm id}_{A}})$ and $\eta'_{A}:=\eta'_{A,A}(\underline{{\rm id}_{A}})$. Since $\overline{\mathbb{E}_{\mathcal{J}}}(\tau C,\tau'A)\simeq \mathbb{D}(\mathcal{C}/\mathcal{P}_{\mathbb{E}_{\mathcal{J}}})(A,\tau C)$, it is equivalent to
	\[\eta'_{A}([f]^{*}(\alpha_{A})_{*}\phi_{C,A}(\delta))=\eta'_{A}([f]^{*}(\alpha_{C})^{*}\phi'_{C,A}(\delta))\text{ for any }[f]\in \mathcal{C}/\mathcal{P}_{\mathbb{E}_{\mathcal{J}}}(A,\tau C).\]
	Since we have
	\begin{align*}
		\eta'_{A}([f]^{*}(\alpha_{A})_{*}\phi_{C,A}(\delta)) & = \eta_{A}([f]^{*}\phi_{C,A}(\delta)) && \text{ by the definition of } \alpha \text{, see Proposition \ref{uniqueness of right ARS}}\\
		& = \eta_{C}([f]_{*}\delta) && \text{ by the definition of }\phi_{C,A}\text{, see Theorem \ref{main thm2}}
	\end{align*}
	and similarly $\eta'_{A}([f]^{*}(\alpha_{C})^{*}\phi'_{C,A}(\delta))=\eta'_{C}((\alpha_{C})_{*}[f]_{*}\delta)=\eta_{C}([f]_{*}\delta)$, the assertion follows.
\end{proof}

\begin{cor}\label{exact equi}
	Under the setting of Theorem \ref{main thm2}. Suppose there is a left ARS duality $(\sigma,\zeta)$ on $(\mathcal{C},\mathbb{E},\mathfrak{s})$. Then there exists a natural isomorphism $\psi:\overline{\mathbb{E}_{\mathcal{J}}}\xRightarrow{\sim} \underline{\mathbb{E}^{\mathcal{Q}}}\circ (\sigma^{\rm op}\times \sigma)$ such that $(\sigma,\psi): (\overline{\mathcal{C}},\overline{\mathbb{E}_{\mathcal{J}}},\overline{\mathfrak{s}|_{\mathbb{E}_{\mathcal{J}}}})\rightarrow (\underline{\mathcal{C}},\underline{\mathbb{E}^{\mathcal{Q}}},\underline{\mathfrak{s}|_{\mathbb{E}^{\mathcal{Q}}}})$ is an exact equivalence. Moreover it is a quasi-inverse of $(\tau,\phi)$.
\end{cor}
\begin{proof}
	The first part follows from the dual of Theorem \ref{main thm2} and Proposition \ref{left & right ARS duality}, we only show the second part. By Proposition \ref{left & right ARS duality}, there are natural isomorphisms $\alpha: {\rm id}_{\overline{\mathcal{C}}}\simeq \tau \sigma$ and $\beta: {\rm id}_{\overline{\mathcal{C}}}\simeq \tau \sigma$. For any $\delta \in \overline{\mathbb{E}_{\mathcal{J}}}(C,A)$, we need to show $(\alpha_{A})_{*}\delta=(\alpha_{C})^{*}\phi\psi(\delta)$. Let $\eta_{A}:=\eta_{A,A}(\underline{{\rm id}_{A}})$ and $\zeta_{A}:=\zeta_{A,A}(\overline{{\rm id}_{A}})$. By the functorial isomorphism $\overline{\mathbb{E}_{\mathcal{J}}}(C,\tau \sigma A)\simeq \mathbb{D}(\mathcal{C}/\mathcal{P}_{\mathbb{E}_{\mathcal{J}}})(\sigma A,C)$, it is equivalent to
	\[\eta_{\sigma A}([f]^{*}(\alpha_{A})_{*}\delta)=\eta_{\sigma A}([f]^{*}(\alpha_{C})^{*}\phi\psi(\delta))\text{ for any }[f]\in \mathcal{C}/\mathcal{P}_{\mathbb{E}_{\mathcal{J}}}(\sigma A,C).\]
	By the definition of $\alpha$, we have $\eta_{\sigma A}([f]^{*}(\alpha_{A})_{*}\delta)=\zeta_{A}([f]^{*}\delta)$. For the right hand side, we have
	\begin{align*}
		\eta_{\sigma A}([f]^{*}(\alpha_{C})^{*}\phi\psi(\delta)) & = \eta_{\sigma C}((\alpha_{C})_{*}[f]_{*}\psi(\delta)) && \text{ by the definition of } \phi \text{, see Theorem \ref{main thm2}} \\
		& = \zeta_{C}([f]_{*}\psi(\delta)) && \text{ by the definition of } \alpha \text{, see Proposition \ref{left & right ARS duality}} \\
		& = \zeta_{A}([f]^{*}\delta), && \text{ by the definition of } \psi \text{, which is similar to } \phi
	\end{align*}
	which proves the identity. The other half is similar.
\end{proof}

Next we show that the functor $\tau$ preserves almost split sequences. 

\begin{defn}\cite[Definition 2.1]{INP}
	A non-split $\mathbb{E}$-extension $\delta \in \mathbb{E}(C,A)$ is said to be {\em almost split} if it satisfies the following conditions
	 \begin{enumerate}
	 	\item [(AS1)] $a_{*}\delta=0$ for any non-section $a\in \mathcal{C}(A,A')$.
	 	\item [(AS2)] $c^{*}\delta=0$ for any non-retraction $c\in \mathcal{C}(C',C)$.
	 \end{enumerate}
\end{defn}

An $\mathfrak{s}$-triangle $A\rightarrow B\rightarrow C\stackrel{\delta}\dashrightarrow $ is called an {\em almost split $\mathfrak{s}$-triangle} if $\delta \in \mathbb{E}(C,A)$ is an almost split extension.

\begin{prop}\label{tau preserve almost split extension}
	Under the setting of Theorem \ref{main thm2}. Suppose $\delta \in \underline{\mathbb{E}^{\mathcal{Q}}}(C,A)$ is an almost split extension in $(\underline{\mathcal{C}},\underline{\mathbb{E}^{\mathcal{Q}}},\underline{\mathfrak{s}|_{\mathbb{E}^{\mathcal{Q}}}})$. Then $\phi_{C,A}(\delta)\in \overline{\mathbb{E}_{\mathcal{J}}}(\tau C,\tau A)$ is also an almost split extension in $(\overline{\mathcal{C}},\overline{\mathbb{E}_{\mathcal{J}}},\overline{\mathfrak{s}|_{\mathbb{E}_{\mathcal{J}}}})$. Thus if $A\xrightarrow{\underline{x}} B\xrightarrow{\underline{y}} C\stackrel{\delta}\dashrightarrow $ is an almost split $\underline{\mathfrak{s}|_{\mathbb{E}^{\mathcal{Q}}}}$-triangle, then
	\[\begin{tikzcd}
		\tau A \arrow[r,"\tau(\underline{x})"] & \tau B \arrow[r,"\tau(\underline{y})"] & \tau C \arrow[r,"\phi_{C,A}(\delta)",dashed] & {}
	\end{tikzcd}\]
	is also an almost split  $\overline{\mathfrak{s}|_{\mathbb{E}_{\mathcal{J}}}}$-triangle.
\end{prop}
\begin{proof}
	Let $\overline{a}\in \overline{\mathcal{C}}(\tau A,X)$ be a non-section. We show that $\overline{a}_{*}\phi_{C,A}(\delta)=0$. Recall $\eta_{A}:=\eta_{A,A}(\underline{{\rm id}_{A}})$. Since $\overline{a}_{*}\phi_{C,A}(\delta) \in \overline{\mathbb{E}_{\mathcal{J}}}(\tau C,X)=\mathbb{E}_{\mathcal{J}}(\tau C,X) \cong \mathbb{D}(\mathcal{C}/\mathcal{I}_{\mathbb{E}_{\mathcal{J}}}(X,\tau^{2}C))$, it suffices to show that for any $\widetilde{f} \in (\mathcal{C}/\mathcal{I}_{\mathbb{E}_{\mathcal{J}}})(X,\tau^{2}C)$ we have $\eta_{\tau C}(\widetilde{f}_{*}\overline{a}_{*}\phi_{C,A}(\delta))=0$. Since $\tau$ is fully faithful, there exists $\underline{g} \in \underline{\mathcal{C}}(A,\tau C)$ satisfying $\tau(\underline{g})=\overline{fa}$. Thus $\widetilde{f}_{*}\overline{a}_{*}\phi_{C,A}(\delta)=\tau(\underline{g})_{*}\phi_{C,A}(\delta)=\phi_{C,\tau C}(\underline{g}_{*}\delta)$. We claim $\underline{g}$ is a non-section. Assume otherwise, then $\overline{a}$ is a section, a contradiction. Since $\delta$ is almost split, $\underline{g}_{*}\delta=0$, which proves (AS1).
	
	As for (AS2), let $\overline{c}\in \overline{\mathcal{C}}(X,\tau C)$ be a non-retraction. We show that $\overline{c}^{*}\phi_{C,A}(\delta)=0$. By Theorem \ref{main thm1}, the right ARS duality $(\tau,\eta)$ induces the right ARS duality $(\tau_{\mathbb{E}^{\mathcal{Q}}},\eta^{\mathbb{E}^{\mathcal{Q}}})$ on $(\mathcal{C},\mathbb{E}^{\mathcal{Q}},\mathfrak{s}|_{\mathbb{E}^{\mathcal{Q}}})$. Thus the ideal quotient $(\underline{\mathcal{C}},\underline{\mathbb{E}^{\mathcal{Q}}},\underline{\mathfrak{s}|_{\mathbb{E}^{\mathcal{Q}}}})$ also has a right ARS duality. By \cite[Proposition 3.3, 2.6]{INP}, there is an isomorphism $\underline{b}\in \underline{\mathcal{C}}(\tau C,A)$. Since $\overline{c}^{*}\phi_{C,A}(\delta) \in \overline{\mathbb{E}_{\mathcal{J}}}(X,\tau A)=\mathbb{E}_{\mathcal{J}}(X,\tau A)\cong \mathbb{D}(\mathcal{C}/\mathcal{P}_{\mathbb{E}_{\mathcal{J}}}(A,X))$, it suffices to show that for any $[f']\in \mathcal{C}/\mathcal{P}_{\mathbb{E}_{\mathcal{J}}}(A,X)$ we have $\eta_{A}([f']^{*}\overline{c}^{*}\phi_{C,A}(\delta))=0$. Consider $\overline{cf'b}\in \overline{\mathcal{C}}(\tau C,\tau C)$, there exists $\underline{g'}\in \underline{\mathcal{C}}(C,C)$ satisfying $\tau(\underline{g'})=\overline{cf'b}$. As above, the morphism $\underline{g'}$ is a non-retraction. Recall that $\mathcal{P}_{\mathbb{E}_{\mathcal{J}}}=[{\rm add}(\mathcal{Q}\cup \mathcal{J})]$. Thus the morphism $[b]\in \mathcal{C}/\mathcal{P}_{\mathbb{E}_{\mathcal{J}}}(\tau C,A)$, which is the image of $\underline{b}\in \underline{\mathcal{C}}(\tau C,A)$, is also an isomorphism. We have $[b]^{*}[f']^{*}\overline{c}^{*}\phi_{C,A}(\delta)=(\overline{cf'b})^{*}\phi_{C,A}(\delta)=\tau(\underline{g'})^{*}\phi_{C,A}(\delta)=\phi_{C,A}({\underline{g'}}^{*}\delta)=0$. Therefore $[f']^{*}\overline{c}^{*}\phi_{C,A}(\delta)=0$ and (AS2) follows.
	
	The last part follows from Theorem \ref{main thm2}.
\end{proof}

\section{Applications and examples}

In this section, let $R$ be a commutative artin ring and $(\mathcal{C},\mathbb{E},\mathfrak{s})$ be a weakly idempotent complete $R$-linear Ext-finite extriangulated category. Suppose $(\mathcal{C},\mathbb{E},\mathfrak{s})$ has enough projectives $\mathcal{Q}$ and enough injectives $\mathcal{J}$ which are both functorially finite. 

\subsection{Induced triangulated structure}

\begin{lem}\label{substructure}
	Suppose $(\mathcal{C},\mathbb{E},\mathfrak{s})$ has an ARS duality $(\tau,\eta)$.
	\begin{enumerate}
		\item $(\mathcal{C},\mathbb{E}^{\mathcal{Q}},\mathfrak{s}|_{\mathbb{E}^{\mathcal{Q}}})$ has enough projectives and enough injectives which are both functorially finite.
		\item $(\mathcal{C},\mathbb{E}_{\mathcal{J}},\mathfrak{s}|_{\mathbb{E}_{\mathcal{J}}})$ has enough projectives and enough injectives which are both functorially finite.
	\end{enumerate}
\end{lem}
\begin{proof}
	We only show (2) since (1) is similar. By part (1) of the proof of Theorem \ref{main thm2}, we know that $(\mathcal{C},\mathbb{E}_{\mathcal{J}},\mathfrak{s}|_{\mathbb{E}_{\mathcal{J}}})$ has enough projectives ${\rm add}(\mathcal{Q}\cup \mathcal{J})$. Since $\mathcal{Q}$ and $\mathcal{J}$ are functorially finite, so is ${\rm add}(\mathcal{Q}\cup \mathcal{J})$. Since $(\mathcal{C},\mathbb{E}^{\mathcal{Q}},\mathfrak{s}|_{\mathbb{E}^{\mathcal{Q}}})$ has enough injectives ${\rm add}(\mathcal{Q}\cup \mathcal{J})$ and $\mathcal{C}$ is weakly idempotent complete, the ideal quotient $(\underline{\mathcal{C}},\underline{\mathbb{E}^{\mathcal{Q}}},\underline{\mathfrak{s}|_{\mathbb{E}^{\mathcal{Q}}}})$ has enough injectives ${\rm add}\mathcal{J}$ which is functorially finite in $\underline{\mathcal{C}}$. By Theorem \ref{main thm2}, there is an exact equivalence $(\tau,\phi):(\underline{\mathcal{C}},\underline{\mathbb{E}^{\mathcal{Q}}},\underline{\mathfrak{s}|_{\mathbb{E}^{\mathcal{Q}}}})\rightarrow (\overline{\mathcal{C}},\overline{\mathbb{E}_{\mathcal{J}}},\overline{\mathfrak{s}|_{\mathbb{E}_{\mathcal{J}}}})$. Thus $(\overline{\mathcal{C}},\overline{\mathbb{E}_{\mathcal{J}}},\overline{\mathfrak{s}|_{\mathbb{E}_{\mathcal{J}}}})$ has enough injectives $\tau({\rm add}\mathcal{J})={\rm add}(\tau \mathcal{J})$ which is functorially finite in $\overline{\mathcal{C}}$. Then $(\mathcal{C},\mathbb{E}_{\mathcal{J}},\mathfrak{s}|_{\mathbb{E}_{\mathcal{J}}})$ has enough injectives ${\rm add}(\mathcal{J}\cup \tau \mathcal{J})$. It remains to show it is contravariantly finite. For $X\in \mathcal{C}$, consider the morphism $\begin{pmatrix}
		a & b
	\end{pmatrix}:J\oplus \tau J'\rightarrow X$ where $a$ is a right $\mathcal{J}$-approximation in $\mathcal{C}$ and $b$ is a morphsim such that $\overline{b}$ is a right ${\rm add}(\tau \mathcal{J})$-approximation in $\overline{\mathcal{C}}$. We claim it is a right ${\rm add}(\mathcal{J}\cup \tau \mathcal{J})$-approximation in $\mathcal{C}$. Let $c: Y\rightarrow X$ be a morphism. It suffices to consider the case where $Y\in \mathcal{J}$ and $Y\in \tau \mathcal{J}$. Since the former one is obvious, we assume $Y\in \tau \mathcal{J}$. There exists $\overline{d}\in \overline{\mathcal{C}}(Y,\tau J')$ satisfying $\overline{c}=\overline{bd}$. Then $c-bd=ae$ for some $e:Y\rightarrow J$ because $a$ is an approximation. Thus we obtain $c= \begin{pmatrix}
	a & b
	\end{pmatrix} \begin{pmatrix}
	e\\d
	\end{pmatrix}$ and the assertion follows.
\end{proof}

\begin{defn}\label{preproj & preinj}
	Suppose $(\mathcal{C},\mathbb{E},\mathfrak{s})$ has an ARS duality $(\tau,\eta)$. Denote by $\widetilde{\mathcal{P}}$ (resp. $\widetilde{\mathcal{I}}$) the subcategory ${\rm add}(\bigcup \limits_{n\geq 0}\tau^{-n}\mathcal{Q})$ (resp. ${\rm add}(\bigcup \limits_{n\geq 0}\tau^{n}\mathcal{J})$). A nonzero object in $\mathcal{C}$ is called {\em preprojective} (resp. {\em preinjective}) if it belongs to $\widetilde{\mathcal{P}}$ (resp. $\widetilde{\mathcal{I}}$). 
\end{defn}

\begin{rem}\label{independent}
	By Proposition \ref{uniqueness of right ARS}, Definition \ref{preproj & preinj} is independent of the choice of $\tau$.
\end{rem}

\begin{prop}\label{induced tri cat}
	Suppose $(\mathcal{C},\mathbb{E},\mathfrak{s})$ has an ARS duality $(\tau,\eta)$ and is Krull-Schmidt and both ${\rm ind}\mathcal{Q}$ and ${\rm ind}\mathcal{J}$ are finite sets. If the preprojectives coincides with the preinjectives, then $\widetilde{\mathcal{C}}:=\mathcal{C}/[\widetilde{\mathcal{P}}]$ is a triangulated category and $\tau$ induces a triangulated auto-equivalence on $\widetilde{\mathcal{C}}$.
\end{prop}
\begin{proof}
	Note that for any $X\in \mathcal{C}$, $\tau X$ may have injective direct summands. Thus for convenience, we construct another $\tau'$ as follows. For any $X\in \mathcal{C}$, define $\tau'X \in \mathcal{C}$ to be a direct summand of $\tau X$ such that $\tau X\cong \tau'X$ in $\overline{\mathcal{C}}$ and $\tau'X$ has no nonzero direct summands in $\mathcal{J}$. Fix an isomorphism $\alpha_{X}\in \overline{\mathcal{C}}(\tau X,\tau'X)$. For any $\underline{f}\in \underline{\mathcal{C}}(X,Y)$, define $\tau'(\underline{f})=\alpha_{Y}\tau(\underline{f})(\alpha_{X})^{-1}$. Then we obtain a functor $\tau'$ and $\alpha:\tau \simeq \tau'$. By Proposition \ref{uniqueness of right ARS}, there is another ARS duality $(\tau',\eta')$. We may replace $(\tau,\eta)$ with $(\tau',\eta')$ at the beginning by Remark \ref{independent}, but keep the notations unchanged.
	
	Since ${\rm ind}\mathcal{Q}$ and ${\rm ind}\mathcal{J}$ are finite, denote their elements by $Q_{1},\cdots,Q_{m}$ and $J_{1},\cdots ,J_{n}$ respectively. By assumptions, for any $Q_{i}$, there exists a non-negative integer $r_{i}$ such that $Q_{i}\cong \tau^{r_{i}}J_{n_{i}}$ for some $J_{n_{i}}$. Then $\tau J_{n_{i}}, \cdots, \tau^{r_{i}-1}J_{n_{i}}$ are all indecomposable non-projective and non-injective. We claim $J_{n_{i}}$ is uniquely determined by $Q_{i}$. Indeed, suppose $Q_{i}\cong \tau^{r_{i}^{'}}J_{n_{i}^{'}}$. If $r_{i}\neq r_{i}^{'}$, suppose $r_{i}>r_{i}^{'}$, then $\tau^{r_{i}-r_{i}^{'}}J_{n_{i}}=J_{n_{i}^{'}}$, a contradiction. Therefor $r_{i}=r_{i}^{'}$ and consequently $J_{n_{i}}\cong J_{n_{i}^{'}}$. If $Q_{i}\ncong Q_{j}$, then clearly $J_{n_{i}}\ncong J_{n_{j}}$. This implies $m\leq n$. Similarly $n\leq m$ holds. Thus $m=n$.
	
	By the argument above, we see that the subcategory
	\[{\rm add}\{\tau^{k_{i}}J_{i}\,|\,1\leq i\leq n,0\leq k_{i}\leq r_{i}\}\]
	coincides with $\widetilde{\mathcal{P}}$ and $\widetilde{\mathcal{I}}$. Let $N:={\rm max}\{r_{1},\cdots,r_{n}\}$. Consider the following construction. Let $\mathbb{E}_{0}=\mathbb{E}$ and $\mathcal{J}_{0}=\mathcal{J}$. Define $\mathbb{E}_{1}:=(\mathbb{E}_{0})_{\mathcal{J}_{0}}$. In the extriangulated category $(\mathcal{C},\mathbb{E}_{1},\mathfrak{s}|_{\mathbb{E}_{1}})$, let $\mathcal{J}_{1}$ be the subcategory of injectives. Define $\mathbb{E}_{2}:=(\mathbb{E}_{1})_{\mathcal{J}_{1}}$. Continue this procedure, we obtain a descending chain of closed subfunctors of $\mathbb{E}$ as follows
	\begin{equation}\label{descend chain}
		\mathbb{E}=\mathbb{E}_{0}\supseteq \mathbb{E}_{1}\supseteq \mathbb{E}_{2}\supseteq \cdots \supseteq \mathbb{E}_{i}\supseteq \cdots
	\end{equation}
	where $\mathbb{E}_{i+1}:=(\mathbb{E}_{i})_{\mathcal{J}_{i}}$ and $i\geq 0$. By Lemma \ref{substructure} and Theorem \ref{main thm1}, $(\mathcal{C},\mathbb{E}_{i},\mathfrak{s}|_{\mathbb{E}_{i}})$ has enough projectives ${\rm add}(\mathcal{Q} \cup \bigcup \limits_{n=0}^{i-1}\tau^{n}\mathcal{J})$ and enough injectives ${\rm add}(\bigcup \limits_{n=0}^{i}\tau^{n}\mathcal{J})$. Therefore the chain (\ref{descend chain}) stabilizes at $N$ and $(\mathcal{C},\mathbb{E}_{N},\mathfrak{s}|_{\mathbb{E}_{N}})$ is a Frobenius extriangulated category, whose subcategory of projectives(=injectives) is exactly $\widetilde{\mathcal{P}}$. By \cite[Corollary 7.4, Remark 7.5]{NP}, the ideal quotient $\widetilde{\mathcal{C}}=\mathcal{C}/[\widetilde{\mathcal{P}}]$ is a triangulated category. By Theorem \ref{main thm1} and \ref{main thm2}, $\tau$ induces an exact auto-equivalence on $\widetilde{\mathcal{C}}$. In this case, by Remark \ref{rem on exact functors} (2), it is just a triangulated auto-equivalence.
\end{proof}

We apply Proposition \ref{induced tri cat} to artin algebras and obtain the following corollary, which might be known to experts. Note that in \cite{ASo}, the authors discussed the properties of the subfunctors of ${\rm Ext}_{\Lambda}^{1}$. From their results, we can also deduce this corollary.

\begin{cor}\label{artin alg case}
	Let $\Lambda$ be an artin algebra. If the $\tau$-orbits of the indecomposable projecrive modules contain all indecomposable injective modules. Then the ideal quotient $\widetilde{{\rm mod}\Lambda}$ of ${\rm mod}\Lambda$ by the additive subcategory generated by these $\tau$-orbits is a triangulated category and $\tau$ induces a triangulated auto-equivalence on it.
\end{cor}

\begin{rem}
	Self-injective algebras clearly satisfy Corollary \ref{artin alg case}. Thus algebras satisfying Corollary \ref{artin alg case} can be regarded as a generalization of self-injective algebra.
\end{rem}

\subsection{Invariance}

Suppose $(\mathcal{C},\mathbb{E},\mathfrak{s})$ has a right ARS duality $(\tau,\eta)$. Let $(\tau,\phi)$ be the exact functor obtained in Theorem \ref{main thm2}. In this section, we generalize a recent result in \cite{H}. Define $\eta_{A}:=\eta_{A,A}(\underline{{\rm id}_{A}})$. By Lemma \ref{equi def}, $\eta$ gives a non-degenerate bilinear form
\[\{-,-\}: \mathbb{E}(C,A)\times \overline{\mathcal{C}}(A,\tau C)\rightarrow J\]
which sends $(\delta,\overline{f})$ to $\{\delta,\overline{f}\}=\eta_{C}(\overline{f}_{*}\delta)$. Now for any $\delta \in \mathbb{E}^{\mathcal{Q}}(C,A) \subseteq \mathbb{E}(C,A)$ and $f\in \mathcal{C}(A,\tau C)$, consider $\phi(\delta)\in \mathbb{E}_{\mathcal{J}}(\tau C,\tau A) \subseteq \mathbb{E}(\tau C,\tau A)$ and $\tau(\underline{f})\in \overline{\mathcal{C}}(\tau A,\tau^{2}C)$. We show the following invariance.

\begin{thm}\label{invariance}
	$\{\delta,\overline{f}\}=\{\phi(\delta),\tau(\underline{f})\}$ holds.
\end{thm}
\begin{proof}
	It suffices to show $\eta_{C}(\overline{f}_{*}\delta)=\eta_{\tau C}(\tau(\underline{f})_{*}\phi(\delta))$. Since we have
	\begin{align*}
		\eta_{\tau C}(\tau(\underline{f})_{*}\phi(\delta)) & = \eta_{A}(\underline{f}^{*}\phi(\delta)) && \text{ by the definition of } \eta \\
		& = \eta_{A}([f]^{*}\phi(\delta)) && [f]\in \mathcal{C}/\mathcal{P}_{\mathbb{E}_{\mathcal{J}}}(A,\tau C)=\mathcal{C}/\mathcal{I}_{\mathbb{E}^{\mathcal{Q}}}(A,\tau C) \text{ is the image of } \underline{f} \text{ and } \overline{f} \\
		& = \eta_{C}([f]_{*}\delta) && \text{ by the definition of } \phi \text{, see Theorem \ref{main thm2}} \\
		& = \eta_{C}(\overline{f}_{*}\delta).
	\end{align*}
	This finishes the proof.
\end{proof}

We apply this theorem to a finite dimensional hereditary $k$-algebra $\Lambda$ ($k$ is a field) and obtain (the dual of) the main result of \cite{H}. Note that in this case, the Auslander-Reiten translation $\tau$ is an endofunctor on ${\rm mod}\Lambda$ and is isomorphic to $\mathbb{D}{\rm Ext}_{\Lambda}^{1}(-,\Lambda)$ (cf. \cite[Ch. VII.1 Corollary 1.9]{ASS}).

\begin{cor}
	For any $C,A\in {\rm mod}\Lambda$ such that $A$ has no nonzero projective direct summands, consider the non-degenerate bilinear form
	\[\{-,-\}:{\rm Ext}_{\Lambda}^{1}(C,A)\times {\rm Hom}_{\Lambda}(A,\tau C)\rightarrow k\]
	given by the Auslander-Reiten formula ${\rm Ext}_{\Lambda}^{1}(C,A)\cong \mathbb{D}{\rm Hom}_{\Lambda}(A,\tau C)$. For any $\zeta \in {\rm Ext}_{\Lambda}^{1}(C,A)$, suppose it is represented by 
	\[0\rightarrow A\xrightarrow{a} B\xrightarrow{b} C\rightarrow 0,\]
	denote by $\tau(\zeta)$ the element in ${\rm Ext}_{\Lambda}^{1}(\tau C,\tau A)$ which is represented by the short exact sequence
	\begin{equation}\label{SES}
		0\rightarrow \tau A\xrightarrow{\tau(a)} \tau B\xrightarrow{\tau(b)} \tau C\rightarrow 0.
	\end{equation}
	Then for any $f\in {\rm Hom}_{\Lambda}(A,\tau C)$ the equality $\{\zeta,f\}=\{\tau(\zeta),\tau(f)\}$ holds.
\end{cor}
\begin{proof}
	We regard ${\rm mod}\Lambda$ as an extriangulated category $({\rm mod}\Lambda,\mathbb{E}:={\rm Ext}_{\Lambda}^{1},{\rm id})$. Since $A$ has no nonzero projective direct summands, ${\rm Hom}_{\Lambda}(A,P)=0$ for any projective $P$. Indeed, for any $g:A\rightarrow P$, the image ${\rm Im}g$ is projective. It is also a projective direct summand of $A$. Therefore ${\rm Im}g=0$. Thus $\zeta \in {\rm Ext}_{\Lambda}^{1}(C,A)$ belongs to $\mathbb{E}^{{\rm add}\Lambda}(C,A)$. Similarly, $\tau(\zeta)\in \mathbb{E}_{{\rm add}\mathbb{D}\Lambda}(\tau C,\tau A)$ holds. Note that we have ${\rm Hom}_{\Lambda}(A,\tau C)=\underline{\rm Hom}_{\Lambda}(A,\tau C)=\overline{\rm Hom}_{\Lambda}(A,\tau C)$ and ${\rm Hom}_{\Lambda}(\tau A,\tau^{2} C)=\overline{\rm Hom}_{\Lambda}(\tau A,\tau^{2} C)$. Thus $f=\overline{f}=\underline{f}$ and $\tau(\underline{f})=\tau(f)\in {\rm Hom}_{\Lambda}(\tau A,\tau^{2}C)$. According to Theorem \ref{invariance}, it suffices to show $\phi(\zeta)\in \mathbb{E}_{{\rm add}\mathbb{D}\Lambda}(\tau C,\tau A)\subseteq {\rm Ext}_{\Lambda}^{1}(\tau C,\tau A)$ is represented by (\ref{SES}). Suppose it is represented by
	\[0\rightarrow \tau A\xrightarrow{c} X\xrightarrow{d} \tau C \rightarrow 0.\]
	By Theorem \ref{main thm2}, there is a commutative diagram in $\overline{\rm mod}\Lambda$ as follows.
	\[\begin{tikzcd}
		\tau A \arrow[r,"\overline{c}"] \arrow[d,equal] & X \arrow[r,"\overline{d}"] \arrow[d,"\cong"] & \tau C \arrow[d,equal] \\
		\tau A \arrow[r,"\overline{\tau(a)}"] & \tau B \arrow[r,"\overline{\tau(b)}"] & \tau C
	\end{tikzcd}\]
	Note that $X\in {\rm mod}\Lambda$ has no nonzero injective direct summands. Indeed, if $I$ is a nonzero injective direct summand of $X$, then it is also a direct summand of $\tau A$, a contradiction. Thus $X\cong \tau B$ in ${\rm mod}\Lambda$. This completes the proof.
\end{proof}

\begin{rem}
	These results indicate that the invariance of the Auslander-Reiten formula under the Auslander-Reiten translation for $\Lambda$ is essentially a categorical property and does not depend on specific algebras. Heredity just simplifies the expressions. 
\end{rem}

\subsection{Examples}\label{examples}

In this subsection we give some examples of Theorem \ref{main thm1}, \ref{main thm2}, Proposition \ref{tau preserve almost split extension} and Corollary \ref{artin alg case}.

\subsubsection*{Hereditary algebra}

Let $\Lambda$ be a finite dimensional hereditary $k$-algebra ($k$ is a field). Denote by ${\rm mod}_{p}\Lambda$ (resp. ${\rm mod}_{i}\Lambda$) the full subcategory of ${\rm mod}\Lambda$ consisting of modules without nonzero projective (resp. injective) direct summands. Then we have canonical inclusion functors
\[{\rm inc}_{p}:{\rm mod}_{p}\Lambda \hookrightarrow {\rm mod}\Lambda \text{ and } {\rm inc}_{i}:{\rm mod}_{i}\Lambda \hookrightarrow {\rm mod}\Lambda.\]
Denote by $\pi_{p}$ (resp. $\pi_{i}$) the canonical projection functor ${\rm mod}\Lambda \rightarrow \underline{\rm mod}\Lambda$ (resp. ${\rm mod}\Lambda \rightarrow \overline{\rm mod}\Lambda$).

\begin{prop}
	The composition of functors $\pi_{p}\circ {\rm inc}_{p}$ and $\pi_{i}\circ {\rm inc}_{i}$ are exact equivalences between exact categories. Thus the exact equivalence $\tau: {\rm mod}_{p}\Lambda \rightarrow {\rm mod}_{i}\Lambda$ implies that $\tau:\underline{\rm mod}\Lambda \rightarrow \overline{\rm mod}\Lambda$ is also an exact equivalence between exact categories.
\end{prop}
\begin{proof}
	Recall that by \cite[Ch IV.1 Proposition 1.15]{ARS}, the two compositions are already equivalences of categories. It remains to show the exactness. First we claim that ${\rm mod}_{p}\Lambda$ is extension-closed in ${\rm mod}\Lambda$. Let
	\[0\rightarrow A\rightarrow B\rightarrow C\rightarrow 0\]
	be a short exact sequence in which $A,C\in {\rm mod}_{p}\Lambda$. Suppose $Q$ is a projective direct summand of $B$. Since ${\rm Hom}_{\Lambda}(A,Q)=0$, $Q$ is also a direct summand of $C$. Thus $Q=0$. Regard ${\rm mod}\Lambda$ as an extriangulated category $({\rm mod}\Lambda, \mathbb{E}, {\rm id})$ where $\mathbb{E}:={\rm Ext}_{\Lambda}^{1}$. By Example \ref{typical exam of subfunctor} and Lemma \ref{induced extri str on ideal quotient}, $\underline{\rm mod}\Lambda$ is also an extriangulated category
	\[(\underline{\rm mod}\Lambda,\underline{\mathbb{E}^{{\rm add}\Lambda}},\underline{\rm id}).\]
	Note that $\mathbb{E}(C,A)=\mathbb{E}^{{\rm add}\Lambda}(C,A)=\underline{\mathbb{E}^{{\rm add}\Lambda}}(C,A)$ for any $C,A\in {\rm mod}_{p}\Lambda$. Then $\pi_{p}\circ {\rm inc}_{p}$ is an exact equivalence between extriangulated categories. Since ${\rm mod}_{p}\Lambda$ is an exact category, so is $\underline{\rm mod}\Lambda$. Similarly, $\pi_{i}\circ {\rm inc}_{i}$ is also an exact equivalence. Since the functor $\tau:{\rm mod}\Lambda \rightarrow {\rm mod}\Lambda$ restricts to an exact equivalence $\tau:{\rm mod}_{p}\Lambda \rightarrow {\rm mod}_{i}\Lambda$, composing it with $\pi_{i}\circ {\rm inc}_{i}$ and a quasi-inverse of $\pi_{p}\circ {\rm inc}_{p}$, we deduce that $\tau:\underline{\rm mod}\Lambda \rightarrow \overline{\rm mod}\Lambda$ is also an exact equivalence.
\end{proof}

\subsubsection*{The category $K^{[-m,0]}({\rm proj}\Lambda)$}

Let $\Lambda$ be a finite dimensional $k$-algebra ($k$ is a field). Denote by $K^{[-m,0]}({\rm proj}\Lambda)$ the full subcategory of $K^{b}({\rm proj}\Lambda)$ consisting of $(m+1)$-term complexes, that is, complexes concentrated in degree $[-m,0]$. It is extension-closed in $K^{b}({\rm proj}\Lambda)$ and thus has a natural extriangulated structure (cf. \cite[Remark 2.18]{NP}) which makes it an $(m-1)$-Auslander extriangulated category (cf. \cite{GNP}) with enough projectives ${\rm add}\Lambda$ and enough injectives $({\rm add}\Lambda)[m]$. Similarly, we have $K^{[-m+1,1]}({\rm inj}\Lambda)$.  There is an equivalence of categories
\[[-1] \circ \nu:K^{[-m,0]}({\rm proj}\Lambda) \rightarrow K^{[-m+1,1]}({\rm inj}\Lambda)\]
where $\nu=\mathbb{D}{\rm Hom}_{\Lambda}(-,\Lambda)$ is the Nakayama functor. In the bounded derived category $D^{b}({\rm mod}\Lambda)$, denote by $(\mathcal{D}^{\leq 0},\mathcal{D}^{\geq 0})$ the canonical $t$-structure and
\[\sigma^{\leq 0}:D^{b}({\rm mod}\Lambda)\rightarrow \mathcal{D}^{\leq 0} \text{ and } \sigma^{\geq 0}:D^{b}({\rm mod}\Lambda)\rightarrow \mathcal{D}^{\geq 0}\]
the canonical truncation functors. Zhou \cite{Zhou} defined the $m$-extended module category $m\text{-mod}\Lambda:=\mathcal{D}^{\leq 0}\cap \mathcal{D}^{\geq -m+1}$. It is an extriangulated category with enough projectives ${\rm proj}\Lambda$ and enough injectives $({\rm inj}\Lambda)[m-1]$ (see \cite[Lemma 3.1]{Zhou}). Gupta \cite[Proposition 3.1]{Gupta} (see also Zhou \cite[Proposition 3.10]{Zhou}) showed that the functor $H^{[-m+1,0]}:=\sigma^{\leq 0}\circ \sigma^{\geq -m+1}$ from $K^{[-m,0]}({\rm proj}\Lambda)$ (resp. $K^{[-m+1,1]}({\rm inj}\Lambda)$) to $m\text{-mod}\Lambda$ induces an equivalence of categories
\begin{equation}\label{equi of cat by Gupta}
	K^{[-m,0]}({\rm proj}\Lambda)/({\rm add}\Lambda)[m]\xrightarrow{\sim} m\text{-mod}\Lambda \,(\text{resp. }K^{[-m+1,1]}({\rm inj}\Lambda)/({\rm add}\nu \Lambda)[-1]\xrightarrow{\sim} m\text{-mod}\Lambda).
\end{equation}
The quasi-inverse is given by taking a minimal projective (resp. injective) presentation (see the proof of \cite[Proposition 3.10]{Zhou}). More precisely, for any $Z^{\bullet} \in m\text{-mod}\Lambda$, denote by $\textbf{p}Z^{\bullet}$ (resp. $\textbf{i}Z^{\bullet}$) its minimal projective (resp. injective) resolution. The minimal projective (resp. injective) presentation is defined by the brutal truncation $\textbf{p}_{m}(Z^{\bullet}):=\sigma_{\geq -m}(\textbf{p}Z^{\bullet})$ (resp. $\textbf{i}_{m}(Z^{\bullet}):=\sigma_{\leq 1}(\textbf{i}Z^{\bullet})$). 

\begin{lem}\label{exact equivalence}
	The equivalences in (\ref{equi of cat by Gupta}) are exact equivalences between extriangulated categories.
\end{lem}
\begin{proof}
	We only show the first one. $K^{[-m,0]}({\rm proj}\Lambda)$ (resp. $m\text{-mod}\Lambda$) is an extriangulated category in which $\mathbb{E}(-,-)={\rm Hom}_{K^{b}({\rm proj}\Lambda)}(-,-[1])$ (resp. $\mathbb{F}(-,-)={\rm Hom}_{D^{b}({\rm mod}\Lambda)}(-,-[1])$). Consider the closed subfunctor $\mathbb{E}_{({\rm add}\Lambda)[m]}$ of $\mathbb{E}$ (see Example \ref{typical exam of subfunctor}) which makes the ideal quotient $K^{[-m,0]}({\rm proj}\Lambda)/({\rm add}\Lambda)[m]$ an extriangulated category by Lemma \ref{induced extri str on ideal quotient}. We define a functorial isomorphism
	\[\phi:\mathbb{E}_{({\rm add}\Lambda)[m]}(Y^{\bullet},X^{\bullet})\rightarrow \mathbb{F}(H^{[-m+1,0]}(Y^{\bullet}),H^{[-m+1,0]}(X^{\bullet})) \text{ for any } Y^{\bullet},X^{\bullet}\in K^{[-m,0]}({\rm proj}\Lambda)\]
	as follows. For any $\delta\in \mathbb{E}_{({\rm add}\Lambda)[m]}(Y^{\bullet},X^{\bullet})$, Consider the following diagram in $D^{b}({\rm mod}\Lambda)$,
	\begin{equation}\label{morphism of triangles}
		\begin{tikzcd}
			\sigma^{\leq -m}(Y^{\bullet}) \arrow[r,"i"] & Y^{\bullet} \arrow[d,"\delta"] \arrow[r] & H^{[-m+1,0]}(Y^{\bullet}) \arrow[r] \arrow[d,dashed,"\phi(\delta)"] & \sigma^{\leq -m}(Y^{\bullet})[1] \\
			\sigma^{\leq -m}(X^{\bullet})[1] \arrow[r] & X^{\bullet}[1] \arrow[r,"p"] & H^{[-m+1,0]}(X^{\bullet})[1] \arrow[r] & \sigma^{\leq -m}(X^{\bullet})[2]
		\end{tikzcd}
	\end{equation}
	in which each row is a triangle. We show the composition $p \circ \delta \circ i$ is zero. Consider chain maps
	\[\begin{tikzcd}
		\cdots \arrow[r] & 0 \arrow[r] \arrow[d] & 0 \arrow[r] \arrow[d] & P \arrow[r] \arrow[d,"c"] \arrow[dddl,dashed,"d" swap] & 0 \arrow[r] \arrow[d] & \cdots \arrow[r] & 0 \arrow[r] \arrow[d] & 0 \arrow[r] \arrow[d] & \cdots \\
		\cdots \arrow[r] & 0 \arrow[r] \arrow[d] & 0 \arrow[r] \arrow[d] & \text{ker}a \arrow[r] \arrow[d,"i^{-m}"] & 0 \arrow[r] \arrow[d] & \cdots \arrow[r] & 0 \arrow[r] \arrow[d] & 0 \arrow[r] \arrow[d] & \cdots \\
		\cdots \arrow[r] & 0 \arrow[r] \arrow[d] & 0 \arrow[r] \arrow[d] & Y^{-m} \arrow[r,"a"] \arrow[d,"\delta^{-m}"] & Y^{-m+1} \arrow[r] \arrow[d] & \cdots \arrow[r] & Y^{-1} \arrow[r] \arrow[d] & Y^{0} \arrow[r] \arrow[d] & \cdots \\
		\cdots \arrow[r] & 0 \arrow[r] \arrow[d] & X^{-m} \arrow[r,"b"] \arrow[d] & X^{-m+1} \arrow[r] \arrow[d,"p^{-m}"] & X^{-m+2} \arrow[r] \arrow[d,equal] & \cdots \arrow[r] & X^{0} \arrow[r] \arrow[d,equal] & 0 \arrow[r] \arrow[d] & \cdots \\
		\cdots \arrow[r] & 0 \arrow[r] & 0 \arrow[r] & \text{coker}b \arrow[r] & X^{-m+2} \arrow[r] & \cdots \arrow[r] & X^{0} \arrow[r] & 0 \arrow[r] & \cdots
	\end{tikzcd}\]
	in which $P \xrightarrow{c} {\rm ker}a$ is a projective cover. By the definition of $\delta$, there exists a morphism $d:P\rightarrow X^{-m}$ satisfying $bd=\delta^{-m} i^{-m} c$. Then we have $p^{-m} \delta^{-m} i^{-m}=0$, which implies $p \circ \delta \circ i=0$. Thus there exists a unique $\phi(\delta)$ making the diagram (\ref{morphism of triangles}) commute. The uniqueness implies the functoriality of $\phi$. Since ${\rm Hom}_{D^{b}({\rm mod}\Lambda)}(Y^{\bullet},\sigma^{\leq -m}(X^{\bullet})[j])=0$ for $j=1,2$, $\phi$ is an isomorphism. Consider a triangle
	\[X^{\bullet} \rightarrow Z^{\bullet} \rightarrow Y^{\bullet} \xrightarrow{\delta} X^{\bullet}[1]\]
	in $K^{b}({\rm proj}\Lambda)$. Combining this triangle with diagram (\ref{morphism of triangles}) and then applying the $3 \times 3$ lemma in $D^{b}({\rm mod}\Lambda)$, we obtain another triangle
	\[H^{[-m+1,0]}(X^{\bullet}) \rightarrow H^{[-m+1,0]}(Z^{\bullet}) \rightarrow H^{[-m+1,0]}(Y^{\bullet}) \xrightarrow{\phi(\delta)} H^{[-m+1,0]}(X^{\bullet})[1]\]
	in $D^{b}({\rm mod}\Lambda)$. This completes the proof.
\end{proof}

By \cite[Example 5.19, Theorem 3.6]{INP}, the category $K^{[-m,0]}({\rm proj}\Lambda)$ has an ARS duality $(\tau,\eta)$. We explicitly give this ARS duality as follows.

\begin{prop}
	For any $P^{\bullet}\in K^{[-m,0]}({\rm proj}\Lambda)$. The assignment
	\[P^{\bullet}\mapsto \boldsymbol{\rm p}_{m}(H^{[-m+1,0]}(\nu P^{\bullet}[-1]))\]
	gives rise to an ARS duality on $K^{[-m,0]}({\rm proj}\Lambda)$. 
\end{prop}
\begin{proof}
	Consider a sequence of functors
	\begin{equation}\label{functor tau}
		\begin{aligned}
			K^{[-m,0]}({\rm proj}\Lambda)/{\rm add}\Lambda & \xrightarrow[\sim]{\nu[-1]} K^{[-m+1,1]}({\rm inj}\Lambda)/({\rm add}\nu \Lambda)[-1] \xrightarrow[\sim]{H^{[-m+1,0]}} m\text{-mod}\Lambda \\
			& \xrightarrow[\sim]{\textbf{p}_{m}}  K^{[-m,0]}({\rm proj}\Lambda)/({\rm add}\Lambda)[m].
		\end{aligned}
	\end{equation}
	Denote the composition by $\tau$, then $\tau$ sends $P^{\bullet}$ to $\textbf{p}_{m}(H^{[-m+1,0]}(\nu P^{\bullet}[-1]))$.
	Moreover for any $P^{\bullet},Q^{\bullet}\in K^{[-m,0]}({\rm proj}\Lambda)$, we have
	\begin{align*}
		{\rm Hom}_{K^{[-m,0]}({\rm proj}\Lambda)/{\rm add}\Lambda}(P^{\bullet},Q^{\bullet}) &\cong {\rm Hom}_{K^{[-m+1,1]}({\rm inj}\Lambda)/({\rm add}\nu \Lambda)[-1]}(\nu P^{\bullet}[-1],\nu Q^{\bullet}[-1])\\
		&\cong {\rm Hom}_{m\text{-mod}\Lambda}(H^{[-m+1,0]}(\nu P^{\bullet}[-1]),H^{[-m+1,0]}(\nu Q^{\bullet}[-1]))\\
		&\cong {\rm Hom}_{D^{b}({\rm mod}\Lambda)}(H^{[-m+1,0]}(\nu P^{\bullet}[-1]),\nu Q^{\bullet}[-1])\\
		&\cong \mathbb{D}{\rm Hom}_{D^{b}({\rm mod}\Lambda)}(Q^{\bullet}[-1],H^{[-m+1,0]}(\nu P^{\bullet}[-1]))\\
		&\cong \mathbb{D}{\rm Hom}_{D^{b}({\rm mod}\Lambda)}(Q^{\bullet}[-1],\textbf{p}_{m}(H^{[-m+1,0]}(\nu P^{\bullet}[-1])))\\
		&\cong \mathbb{D}{\rm Hom}_{K^{b}({\rm proj}\Lambda)}(Q^{\bullet},\tau(P^{\bullet})[1]).
	\end{align*}
	We explain the third and the fifth isomorphism. In $D^{b}({\rm mod}\Lambda)$, consider the triangle
	\[H^{[-m+1,0]}(\nu Q^{\bullet}[-1])\rightarrow \nu Q^{\bullet}[-1]\rightarrow \sigma^{\geq 1}(\nu Q^{\bullet}[-1])\rightarrow H^{[-m+1,0]}(\nu Q^{\bullet}[-1])[1].\]
	Applying the functor ${\rm Hom}_{D^{b}({\rm mod}\Lambda)}(H^{[-m+1,0]}(\nu P^{\bullet}[-1]),-)$ to this triangle, then the third isomorphism follows. Since $H^{[-m+1,0]}(\nu P^{\bullet}[-1])$ and $\textbf{p}(H^{[-m+1,0]}(\nu P^{\bullet}[-1]))$ are isomorphic in $D^{b}({\rm mod}\Lambda)$, we consider the triangle in $K^{-}({\rm proj}\Lambda)$
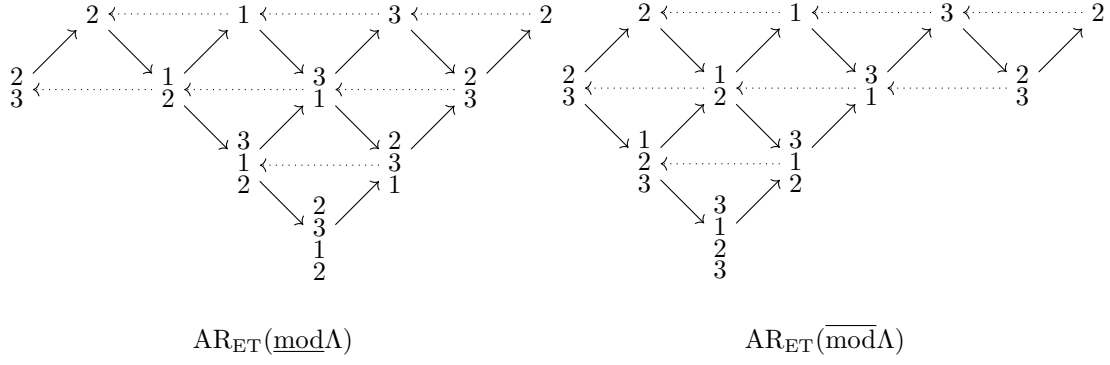
	\begin{align*}
		\textbf{p}_{m}(H^{[-m+1,0]}(\nu P^{\bullet}[-1])) & \rightarrow \textbf{p}(H^{[-m+1,0]}(\nu P^{\bullet}[-1])) \rightarrow \sigma_{\leq -m-1}(\textbf{p}(H^{[-m+1,0]}(\nu P^{\bullet}[-1])))\\
		& \rightarrow \textbf{p}_{m}(H^{[-m+1,0]}(\nu P^{\bullet}[-1]))[1].
	\end{align*}
	Applying the functor ${\rm Hom}_{K^{-}({\rm proj}\Lambda)}(Q^{\bullet}[-1],-)$ to this triangle, then the fifth isomorphism follows. Denote by $\eta_{P^{\bullet},Q^{\bullet}}$ the composed isomorphism, it is clearly functorial. Thus we obtain an ARS duality $(\tau,\eta)$ on $K^{[-m,0]}({\rm proj}\Lambda)$.
\end{proof}

The first functor $\nu[-1]$ in (\ref{functor tau}) is clearly an exact equivalence between extriangulated categories. By Lemma \ref{exact equivalence}, the second one is also exact. Since the functor  $\textbf{p}_{m}$ in (\ref{functor tau}) is a quasi-inverse of the exact equivalence $H^{[-m+1,0]}: K^{[-m,0]}({\rm proj}\Lambda)/({\rm add}\Lambda)[m]\xrightarrow{\sim} m\text{-mod}\Lambda$, it is also exact (cf. \cite[Proposition 2.13]{NOS}). Therefore the functor $\tau$ is an exact equivalence.

By Lemma \ref{exact equivalence} and Theorem \ref{main thm1}, $(\tau,\eta)$ induces an ARS duality on $m\text{-mod}\Lambda$. We show it is exactly the Auslander-Reiten translation $\tau_{[m]}$ defined in \cite[Definition 3.7]{Zhou}. There is a commutative diagram
\[\begin{tikzcd}
	K^{[-m,0]}({\rm proj}\Lambda)/{\rm add}\Lambda \arrow[r,"\tau"] \arrow[d,"H^{[-m+1,0]}",swap] \arrow[dr, phantom, "\circlearrowright"] & K^{[-m,0]}({\rm proj}\Lambda)/({\rm add}\Lambda)[m] \arrow[d,"H^{[-m+1,0]}"]\\
	\underline{m\text{-mod}}\Lambda \arrow[r,"\tau_{[m]}",swap]& \overline{m\text{-mod}}\Lambda
\end{tikzcd}\]
in which the vertical functors are induced by $H^{[-m+1,0]}:K^{[-m,0]}({\rm proj}\Lambda)\rightarrow m\text{-mod}\Lambda$. This is essentially the first diagram in Theorem \ref{main thm1}. Thus the Auslander-Reiten translation $\tau_{[m]}$ can be deduced from the ARS duality $(\tau,\eta)$ on $K^{[-m,0]}({\rm proj}\Lambda)$.

\subsubsection*{Bounded quiver algebra}

Let $k$ be an algebraically closed field and $\Lambda$ be a $k$-algebra given by the quiver $Q$
\[\begin{tikzcd}
	& 2 \arrow[dr,"\beta"] & \\
	1 \arrow[ur,"\alpha"] & & 3 \arrow[ll,"\gamma"]
\end{tikzcd}\]
bounded by $\alpha \beta \gamma=0$. Then the Auslander-Reiten quiver $\Gamma({\rm mod}\Lambda)$ is as follows.
\begin{figure}[htbp]
	\renewcommand*{\arraystretch}{0.5}
	\centering
	\begin{tikzpicture}
		\node (1) at (1,4) {2};
		\node (2) at (3,4) {1};
		\node (3) at (5,4) {3};
		\node (4) at (7,4) {2};
		\node (5) at (0,3) {$\begin{matrix}
				2\\3
		\end{matrix}$};
		\node (6) at (2,3) {$\begin{matrix}
				1\\2
			\end{matrix}$};
		\node (7) at (4,3) {$\begin{matrix}
				3\\1
			\end{matrix}$};
		\node (8) at (6,3) {$\begin{matrix}
				2\\3
			\end{matrix}$};
		\node (9) at (1,2) {$\begin{matrix}
				1\\2\\3
			\end{matrix}$};
		\node (10) at (3,2) {$\begin{matrix}
				3\\1\\2
			\end{matrix}$};
		\node (11) at (5,2) {$\begin{matrix}
				2\\3\\1
			\end{matrix}$};
		\node (12) at (2,1) {$\begin{matrix}
				3\\1\\2\\3
			\end{matrix}$};
		\node (13) at (4,1) {$\begin{matrix}
				2\\3\\1\\2
			\end{matrix}$};
		\node (14) at (3,0) {$\begin{matrix}
				2\\3\\1\\2\\3
			\end{matrix}$};
		\draw[->] (1)--(6);
		\draw[->,dotted] (2)--(1);
		\draw[->] (2)--(7);
		\draw[->,dotted] (3)--(2);
		\draw[->] (3)--(8);
		\draw[->,dotted] (4)--(3);
		\draw[->] (5)--(1);
		\draw[->] (5)--(9);
		\draw[->,dotted] (6)--(5);
		\draw[->] (6)--(10);
		\draw[->] (6)--(2);
		\draw[->,dotted] (7)--(6);
		\draw[->] (7)--(11);
		\draw[->] (7)--(3);
		\draw[->,dotted] (8)--(7);
		\draw[->] (8)--(4);
		\draw[->] (9)--(6);
		\draw[->] (9)--(12);
		\draw[->,dotted] (10)--(9);
		\draw[->] (10)--(7);
		\draw[->] (10)--(13);
		\draw[->,dotted] (11)--(10);
		\draw[->] (11)--(8);
		\draw[->] (12)--(10);
		\draw[->] (12)--(14);
		\draw[->,dotted] (13)--(12);
		\draw[->] (13)--(11);
		\draw[->] (14)--(13);
	\end{tikzpicture}
	\caption*{$\Gamma({\rm mod}\Lambda)$}
\end{figure}

Let $\mathbb{E}:={\rm Ext}_{\Lambda}^{1}$. Regard $\underline{\rm mod}\Lambda$ (resp. $\overline{\rm mod}\Lambda$) as an extriangulated category $(\underline{\rm mod}\Lambda,\underline{\mathbb{E}^{{\rm add}\Lambda}},\underline{\rm id})$ (resp. $(\overline{\rm mod}\Lambda,\overline{\mathbb{E}_{{\rm add}\nu \Lambda}},\overline{\rm id})$). The Auslander-Reiten quivers ${\rm AR}_{\rm ET}(\underline{\rm mod}\Lambda)$ and ${\rm AR}_{\rm ET}(\overline{\rm mod}\Lambda)$ are as follows (cf. \cite[Section 3.3]{INP}).
\begin{figure}[htbp]
	\renewcommand*{\arraystretch}{0.5}
	\centering
	\begin{minipage}[b]{0.49\linewidth}
		\begin{tikzpicture}
			\node (1) at (1,4) {2};
			\node (2) at (3,4) {1};
			\node (3) at (5,4) {3};
			\node (4) at (7,4) {2};
			\node (5) at (0,3) {$\begin{matrix}
					2\\3
				\end{matrix}$};
			\node (6) at (2,3) {$\begin{matrix}
					1\\2
				\end{matrix}$};
			\node (7) at (4,3) {$\begin{matrix}
					3\\1
				\end{matrix}$};
			\node (8) at (6,3) {$\begin{matrix}
					2\\3
				\end{matrix}$};
			\node (10) at (3,2) {$\begin{matrix}
					3\\1\\2
				\end{matrix}$};
			\node (11) at (5,2) {$\begin{matrix}
					2\\3\\1
				\end{matrix}$};
			\node (13) at (4,1) {$\begin{matrix}
					2\\3\\1\\2
				\end{matrix}$};
			\draw[->] (1)--(6);
			\draw[->,dotted] (2)--(1);
			\draw[->] (2)--(7);
			\draw[->,dotted] (3)--(2);
			\draw[->] (3)--(8);
			\draw[->,dotted] (4)--(3);
			\draw[->] (5)--(1);
			\draw[->,dotted] (6)--(5);
			\draw[->] (6)--(10);
			\draw[->] (6)--(2);
			\draw[->,dotted] (7)--(6);
			\draw[->] (7)--(11);
			\draw[->] (7)--(3);
			\draw[->,dotted] (8)--(7);
			\draw[->] (8)--(4);
			\draw[->] (10)--(7);
			\draw[->] (10)--(13);
			\draw[->,dotted] (11)--(10);
			\draw[->] (11)--(8);
			\draw[->] (13)--(11);
		\end{tikzpicture}
		\caption*{${\rm AR}_{\rm ET}(\underline{\rm mod}\Lambda)$}
	\end{minipage}
	\begin{minipage}[b]{0.49\linewidth}
		\begin{tikzpicture}
			\node (1) at (1,4) {2};
			\node (2) at (3,4) {1};
			\node (3) at (5,4) {3};
			\node (4) at (7,4) {2};
			\node (5) at (0,3) {$\begin{matrix}
					2\\3
				\end{matrix}$};
			\node (6) at (2,3) {$\begin{matrix}
					1\\2
				\end{matrix}$};
			\node (7) at (4,3) {$\begin{matrix}
					3\\1
				\end{matrix}$};
			\node (8) at (6,3) {$\begin{matrix}
					2\\3
				\end{matrix}$};
			\node (9) at (1,2) {$\begin{matrix}
					1\\2\\3
				\end{matrix}$};
			\node (10) at (3,2) {$\begin{matrix}
					3\\1\\2
				\end{matrix}$};
			\node (12) at (2,1) {$\begin{matrix}
					3\\1\\2\\3
				\end{matrix}$};
			\draw[->] (1)--(6);
			\draw[->,dotted] (2)--(1);
			\draw[->] (2)--(7);
			\draw[->,dotted] (3)--(2);
			\draw[->] (3)--(8);
			\draw[->,dotted] (4)--(3);
			\draw[->] (5)--(1);
			\draw[->] (5)--(9);
			\draw[->,dotted] (6)--(5);
			\draw[->] (6)--(10);
			\draw[->] (6)--(2);
			\draw[->,dotted] (7)--(6);
			\draw[->] (7)--(3);
			\draw[->,dotted] (8)--(7);
			\draw[->] (8)--(4);
			\draw[->] (9)--(6);
			\draw[->] (9)--(12);
			\draw[->,dotted] (10)--(9);
			\draw[->] (10)--(7);
			\draw[->] (12)--(10);
		\end{tikzpicture}
		\caption*{${\rm AR}_{\rm ET}(\overline{\rm mod}\Lambda)$}
	\end{minipage}
\end{figure}

In this example, we can explicitly check some of our main results. For instance, the short exact sequence
\begin{equation*}
	\renewcommand*{\arraystretch}{0.5}
	0 \rightarrow \begin{matrix}
		2\\3
	\end{matrix} \rightarrow \begin{matrix}
		1\\2\\3
	\end{matrix} \rightarrow 1 \rightarrow 0
\end{equation*}
in ${\rm mod}\Lambda$ gives rise to an exangle
\begin{equation*}
	\renewcommand*{\arraystretch}{0.5}
	\begin{matrix}
		2\\3
	\end{matrix} \rightarrow 0 \rightarrow 1 \dashrightarrow
\end{equation*}
in $\underline{\rm mod}\Lambda$. The functor $\tau$ sends this exangle to the exangle
\begin{equation*}
	\renewcommand*{\arraystretch}{0.5}
	\begin{matrix}
		3\\1
	\end{matrix} \rightarrow 0 \rightarrow 2 \dashrightarrow
\end{equation*}
in $\overline{\rm mod}\Lambda$, which is induced by the short exact sequence
\begin{equation*}
	\renewcommand*{\arraystretch}{0.5}
	0 \rightarrow \begin{matrix}
		3\\1
	\end{matrix} \rightarrow \begin{matrix}
		2\\3\\1
	\end{matrix} \rightarrow 2 \rightarrow 0
\end{equation*}
in ${\rm mod}\Lambda$. We can also check that $\tau$ sends all the 7 almost split exangles in $\underline{\rm mod}\Lambda$ to all the 7 ones in $\overline{\rm mod}\Lambda$. For instance, the almost split exangle
\begin{equation*}
	\renewcommand*{\arraystretch}{0.5}
	\begin{matrix}
		2\\3
	\end{matrix} \rightarrow 2 \rightarrow \begin{matrix}
	1\\2
	\end{matrix} \dashrightarrow
\end{equation*}
in $\underline{\rm mod}\Lambda$ is sent to the almost split exangle
\begin{equation*}
	\renewcommand*{\arraystretch}{0.5}
	\begin{matrix}
		3\\1
	\end{matrix} \rightarrow 3 \rightarrow \begin{matrix}
		2\\3
	\end{matrix} \dashrightarrow
\end{equation*}
in $\overline{\rm mod}\Lambda$. These two exangles are induced by the almost split short exact sequences
\begin{equation*}
	\renewcommand*{\arraystretch}{0.5}
	0 \rightarrow \begin{matrix}
		2\\3
	\end{matrix} \rightarrow 2 \oplus \begin{matrix}
	1\\2\\3
	\end{matrix} \rightarrow \begin{matrix}
	1\\2
	\end{matrix} \rightarrow 0
\end{equation*}
and
\begin{equation*}
	\renewcommand*{\arraystretch}{0.5}
	0 \rightarrow \begin{matrix}
		3\\1
	\end{matrix} \rightarrow 3 \oplus \begin{matrix}
	2\\3\\1
	\end{matrix} \rightarrow \begin{matrix}
	2\\3
	\end{matrix} \rightarrow 0
\end{equation*}
in ${\rm mod}\Lambda$ respectively.

Moreover, this example satisfies the conditions in Corollary \ref{artin alg case}. Denote by $\mathcal{C}$ the additive subcategory generated by the $\tau$-orbits of all indecomposable projective modules. Regard $\widetilde{{\rm mod}\Lambda}:={\rm mod}\Lambda /[\mathcal{C}]$ as an extriangulated category $(\widetilde{{\rm mod}\Lambda},\widetilde{\mathbb{E}_{\mathcal{C}}\cap \mathbb{E}^{\mathcal{C}}},\widetilde{\rm id})$. Then the Auslander-Reiten quiver ${\rm AR}_{\rm ET}(\widetilde{{\rm mod}\Lambda})$ is as follows.
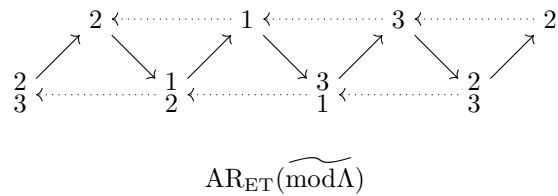
\begin{figure}[htbp]
	\renewcommand*{\arraystretch}{0.5}
	\centering
	\begin{tikzpicture}
		\node (1) at (1,4) {2};
		\node (2) at (3,4) {1};
		\node (3) at (5,4) {3};
		\node (4) at (7,4) {2};
		\node (5) at (0,3) {$\begin{matrix}
				2\\3
			\end{matrix}$};
		\node (6) at (2,3) {$\begin{matrix}
				1\\2
			\end{matrix}$};
		\node (7) at (4,3) {$\begin{matrix}
				3\\1
			\end{matrix}$};
		\node (8) at (6,3) {$\begin{matrix}
				2\\3
			\end{matrix}$};
		\draw[->] (1)--(6);
		\draw[->,dotted] (2)--(1);
		\draw[->] (2)--(7);
		\draw[->,dotted] (3)--(2);
		\draw[->] (3)--(8);
		\draw[->,dotted] (4)--(3);
		\draw[->] (5)--(1);
		\draw[->,dotted] (6)--(5);
		\draw[->] (6)--(2);
		\draw[->,dotted] (7)--(6);
		\draw[->] (7)--(3);
		\draw[->,dotted] (8)--(7);
		\draw[->] (8)--(4);
	\end{tikzpicture}
	\caption*{${\rm AR}_{\rm ET}(\widetilde{{\rm mod}\Lambda})$}
\end{figure}

We now confirm the conclusion of Corollary \ref{artin alg case}. Consider the algebra $A$ given by the same quiver $Q$ bounded by $\alpha \beta \gamma=0, \beta \gamma \alpha=0$ and $\gamma \alpha \beta=0$. Then $A$ is a quotient of $\Lambda$ and is self-injective. The Auslander-Reiten quiver is as follows.
\begin{figure}[htbp]
	\renewcommand*{\arraystretch}{0.5}
	\centering
	\begin{tikzpicture}
		\node (1) at (1,4) {2};
		\node (2) at (3,4) {1};
		\node (3) at (5,4) {3};
		\node (4) at (7,4) {2};
		\node (5) at (0,3) {$\begin{matrix}
				2\\3
			\end{matrix}$};
		\node (6) at (2,3) {$\begin{matrix}
				1\\2
			\end{matrix}$};
		\node (7) at (4,3) {$\begin{matrix}
				3\\1
			\end{matrix}$};
		\node (8) at (6,3) {$\begin{matrix}
				2\\3
			\end{matrix}$};
		\node (9) at (1,2) {$\begin{matrix}
				1\\2\\3
			\end{matrix}$};
		\node (10) at (3,2) {$\begin{matrix}
				3\\1\\2
			\end{matrix}$};
		\node (11) at (5,2) {$\begin{matrix}
				2\\3\\1
			\end{matrix}$};
		\draw[->] (1)--(6);
		\draw[->,dotted] (2)--(1);
		\draw[->] (2)--(7);
		\draw[->,dotted] (3)--(2);
		\draw[->] (3)--(8);
		\draw[->,dotted] (4)--(3);
		\draw[->] (5)--(1);
		\draw[->] (5)--(9);
		\draw[->,dotted] (6)--(5);
		\draw[->] (6)--(10);
		\draw[->] (6)--(2);
		\draw[->,dotted] (7)--(6);
		\draw[->] (7)--(11);
		\draw[->] (7)--(3);
		\draw[->,dotted] (8)--(7);
		\draw[->] (8)--(4);
		\draw[->] (9)--(6);
		\draw[->] (10)--(7);
		\draw[->] (11)--(8);
	\end{tikzpicture}
	\caption*{$\Gamma({\rm mod}A)$}
\end{figure}

It is a part of $\Gamma({\rm mod}\Lambda)$. From the AR quiver we can roughly see that $\widetilde{{\rm mod}\Lambda}$ is equivalent to the stable category of ${\rm mod}A$ and thus triangulated. Indeed, the inclusion ${\rm mod}A\hookrightarrow {\rm mod}\Lambda$ induces an equivalence
\begin{equation}\label{equi}
	\underline{\rm mod}A\xrightarrow{\simeq}\widetilde{{\rm mod}\Lambda}.
\end{equation}
Moreover we have ${\rm Ext}_{A}^{1}(Y,X)=\mathbb{E}_{\mathcal{C}}(Y,X)=\mathbb{E}^{\mathcal{C}}(Y,X)$ for any $A$-modules $X,Y$, which implies that $({\rm mod}\Lambda,\mathbb{E}_{\mathcal{C}}\cap \mathbb{E}^{\mathcal{C}},{\rm id})$ is a Frobenius exact category and thus $\widetilde{{\rm mod}\Lambda}$ is a triangulated category and (\ref{equi}) is a triangulated equivalence. Since the Auslander-Reiten translation $\tau_{\Lambda}$ induces an auto-equivalence $\tau_{\Lambda}^{'}$ on $\widetilde{{\rm mod}\Lambda}$ and there is a commutative diagram
\[\begin{tikzcd}
	\underline{\rm mod}A \arrow[r,"\tau_{A}","\simeq" swap] \arrow[d,"\simeq"] & \underline{\rm mod}A \arrow[d,"\simeq"] \\
	\widetilde{{\rm mod}\Lambda} \arrow[r,"\tau_{\Lambda}^{'}","\simeq" swap] & \widetilde{{\rm mod}\Lambda},
\end{tikzcd}\]
the functor $\tau_{\Lambda}^{'}$ is a triangulated auto-equivalence.

\section*{Acknowledgment}
The authors would like to thank professor Bin Zhu for helpful discussions and suggestions. Ji-Wei He is supported by National Natural Science Foundation of China (No. 12371017). Jixing Pan is supported by National Natural Science Foundation of China (No. 12601066).

\bibliographystyle{plain}

\begin{thebibliography}{100}
	
	\bibitem{ASS}
	I. Assem, D. Simson, A. Skowroński.
	\newblock{\em Elements of the representation theory of associative algebras. Vol. 1. Techniques of representation theory}.
	\newblock{London Math. Soc. Stud. Texts 65 Cambridge University Press Cambridge (2006) x+458 pp}.
	
	\bibitem{ARS}
	M. Auslander, Idun. Reiten, S. O. Smalø.
	\newblock{\em Representation theory of Artin algebras}.
	\newblock{Cambridge Stud. Adv. Math. 36 Cambridge University Press (1995) xiv+423 pp}.
	
	\bibitem{ASo}
	M. Auslander, Ø. Solberg.
	\newblock{\em Relative homology and representation theory. I. Relative homology and homologically finite subcategories}.
	\newblock{Comm. Algebra 21 (1993), no. 9, 2995–3031}.
	
	\bibitem{BS}
	R. Bennett-Tennenhaus, A. Shah.
	\newblock{\em Transport of structure in higher homological algebra}.
	\newblock{J. Algebra 574 (2021) 514–549}.
	
	\bibitem{Bo}
	R. Bocklandt.
	\newblock{\em Graded Calabi Yau algebras of dimension 3, with an appendix by M. Van den Bergh}.
	\newblock{J. Pure Appl. Algebra 212 (2008), no. 1, 14–32}.
	
	\bibitem{BK}
	A. I. Bondal, M. M. Kapranov.
	\newblock{\em Representable functors, Serre functors, and reconstructions}.
	\newblock{Izv. Akad. Nauk SSSR Ser. Mat. 53 (1989), no. 6, 1183–1205}.
	
	\bibitem{Buhler}
	T. Bühler.
	\newblock{\em Exact categories}.
	\newblock{Expo. Math. 28 (1) (2010) 1–69}.
	
	\bibitem{DRSSK}
	P. Dräxler, I. Reiten, S. O. Smalø, Ø. Solberg.
	\newblock{\em Exact categories and vector space categories, with an appendix by B. Keller}.
	\newblock{Trans. Amer. Math. Soc. 351 (1999), no. 2, 647–682}.
	
	\bibitem{GNP}
	M. Gorsky, H. Nakaoka, Y. Palu.
	\newblock{\em Hereditary extriangulated categories: silting objects, mutation, negative extensions}.
	\newblock{aeXiv:2303.07134v2}.
	
	\bibitem{Gupta}
	E. Gupta.
	\newblock{\em On $d$-term silting objects, torsion classes, and cotorsion classes}.
	\newblock{arXiv:2407.10562v3}.
	
	\bibitem{HLN}
	M. Herschend, Y. Liu, H. Nakaoka.
	\newblock{\em $n$-exangulated categories (I): definitions and fundamental properties}.
	\newblock{J. Algebra 570 (2021) 531-586}.
	
	\bibitem{H}
	A. Hubery.
	\newblock{\em The invariance of the Auslander-Reiten formula for hereditary algebras}.
	\newblock{arXiv:2602.16332v1}.
	
	\bibitem{INP}
	O. Iyama, H. Nakaoka, Y. Palu.
	\newblock{\em Auslander-Reiten theory in extriangulated categories}.
	\newblock{Trans. Amer. Math. Soc. Ser. B 11 (2024) 248–305}.
	
	\bibitem{K}
	C. Klapproth.
	\newblock{\em $n$-extension closed subcategories of $n$-exangulated categories}.
	\newblock{arXiv:2209.01128v3}.
	
	\bibitem{LN}
	Y. Liu, H. Nakaoka.
	\newblock{\em Hearts of twin cotorsion pairs on extriangulated categories}.
	\newblock{J. Algebra 528 (2019) 96-149}.
	
	\bibitem{NOS}
	H. Nakaoka, Y. Ogawa, A. Sakai.
	\newblock{\em Localization of extriangulated categories}.
	\newblock{J. Algebra 611 (2022) 341–398}.
	
	\bibitem{NP}
	H. Nakaoka, Y. Palu.
	\newblock{\em Extriangulated categories, Hovey twin cotorsion pairs and model structures}.
	\newblock{Cah. Topol. Géom. Différ. Catég. 60 (2) (2019) 117-193}.
	
	\bibitem{Zhou}
	Y. Zhou.
	\newblock{\em Tilting theory for extended module categories}.
	\newblock{arXiv:2411.15473v2}.
	
\end{thebibliography}

\end{document}